\documentclass[reqno]{amsart}

\usepackage{graphicx}        
\usepackage{multicol}        

\usepackage{amsthm}
\usepackage{amssymb}
\usepackage{xcolor}

\theoremstyle{plain}
\newtheorem{theorem}{Theorem}
\newtheorem{corollary}{Corollary}
\newtheorem{lemma}{Lemma}
\newtheorem{remark}{Remark}

\newcommand{\N}{\mathbb{N}}
\newcommand{\R}{\mathbb{R}}

\newcommand{\metric}{\langle \, , \, \rangle}

\newcommand{\PP}{\mathcal{P}}
\newcommand{\eps}{\varepsilon}

\newcommand{\Sph}{\mathbb{S}}
\newcommand{\di}{\mathrm{d}}

\newcommand{\HH}{\mathbb{H}}

\newcommand{\Ricc}{\mathrm{Ric}}

\newcommand{\vol}{\mathrm{vol}}

\newcommand{\loc}{\mathrm{loc}}

\newcommand{\RR}{\mathbb{R}}

\DeclareMathOperator{\dist}{dist}

\renewcommand{\div}{\mathrm{div}}

\begin{document}

\title[Maximum principles for weakly $1$-coercive operators]{Maximum principles for weakly $1$-coercive operators with applications to capillary and prescribed mean curvature graphs}

\author{Luis J. Al\'{\i}as}
\address{Departamento de Matem\'{a}ticas, Universidad de Murcia, Campus de Espinardo, E-30100 Espinardo, Murcia, Spain}
\email{ljalias@um.es}

\author{Giulio Colombo}
\address{Dipartimento di Matematica e Applicazioni ``R. Caccioppoli'', Universit\`{a} degli Studi di Napoli ``Federico II'', Via Vicinale Cupa Cintia 26, I-80126 Naples, Italy}
\email{giulio.colombo@unina.it}

\author{Marco Rigoli}
\address{Dipartimento di Matematica ``F. Enriques'', Universit\`a degli Studi di Milano, Via Saldini 50, I-20133 Milano, Italy}
\email{marco.rigoli@unimi.it}

\maketitle

\begin{abstract}
	In this paper we establish maximum principles for weakly 1-coercive operators $L$ on complete, non-compact Riemannian manifolds $M$. In particular, we search for conditions under which one can guarantee that solutions $u$ of differential equations of the form $L(u)\geq f(u)$ satisfy $f(u)\leq 0$ on $M$. The case of weakly $p$-coercive operators with $p>1$, including the $p$-Laplacian and in particular the Laplace-Beltrami operator for $p=2$, has been considered in our recent paper \cite{acr24}. As a consequence of the main results we infer comparison principles for that kind of operators. Furthermore we apply them to geometric situations dealing with the mean curvature operator, which is a typical weakly 1-coercive operator. We first consider the case of $\mathcal C^1$ operators $L$ acting on functions $u$ of class $\mathcal C^2$ and, in the last section of the paper, we show how our results can be extended to the case of less regular operators $L$ acting on functions $u$ which are just continuous and locally $W^{1,1}$ regular.
\end{abstract}

\bigskip

\begin{center}
	\textit{Dedicated to Professor Marcos Dajczer on his 75$^{\text{th}}$ birthday}
\end{center}

\section{Introduction}
In what follows, $(M,\metric)$ will always denote a connected Riemannian manifold of dimension $m\geq 2$. In its simplest form, the validity of the weak maximum principle (WMP) for the Laplace-Beltrami operator $\Delta$ on $M$ states that if $u$ is (to simplify the matter) a $\mathcal C^2$ function on $M$ such that
\begin{equation}
	u^* = \sup_M u < +\infty
\end{equation} 
then, there exists a sequence $\{x_k\}\subset M$ with the following properties
\begin{equation}
	u(x_k) > u^* - \frac{1}{k} \quad \text{ and } \quad \Delta u(x_k) < \frac{1}{k}
\end{equation}
for each $k\in\N$. This is equivalent to the stochastic completeness of $M$, see \cite{prs03}, and the latter can be analytically expressed in the following terms: for each $\lambda>0$, the only solution $u\in\mathcal C^2(M)$ of
\begin{equation}
\label{I1.3}
\Delta u\geq \lambda u \quad \text{ on } M
\end{equation}
satisfying
\begin{equation}
0\leq u\leq u^*<+\infty
\end{equation}
is the trivial solution $u\equiv 0$, see \cite{g99}. In the above formulation, allowing a weak interpretation of (\ref{I1.3}), we can avoid the request $u\geq 0$; indeed, for $u$ a solution of (\ref{I1.3}), 
\[
u_{+}=\max\{u,0\}
\]
is still a (weak) solution of  (\ref{I1.3})  on $M$. Note that, since the WMP for $\Delta$ on $(M,\metric)$ is equivalent to stochastic completeness, geodesic completeness is not necessary for its validity, thus the WMP for $\Delta$ holds on $\R^2\setminus\{0\}$ with the Euclidean (uncomplete) metric.

Replacing the right hand side of (\ref{I1.3}) with $f(u)$ for some continuous $f:\R\rightarrow\R$, in case $f$ satisfies
\begin{equation}
\label{I1.5}
\liminf_{t\to+\infty}\frac{f(t)}{t^\sigma}>0
\end{equation}
for some $\sigma>1$, it can be seen that assumptions which guarantee the validity of the WMP for $\Delta$ also guarantee that a solution $u$ of 
\begin{equation}
\label{I1.6}
\Delta u\geq f(u)
\end{equation}
first satisfies $u^*<+\infty$ and then application of the WMP implies $f(u^*)\leq 0$; see Theorems 4.1 and 4.2 of \cite{amr16} for very general operators others than $\Delta$ and further applications. Please note that in condition (4.29) in \cite{amr16}, due to a typo, a ``log''  is missing just before the integral over $B_r$. We emphasize here that assumption (\ref{I1.5}) plays an essential role in establishing $u^\ast<+\infty$. 

Once we assume geodesic completeness of the metric $\metric$, there are at least three interesting types of assumptions that guarantee the validity of the WMP for the Laplace-Beltrami operator (and, in fact, for many other operators). The first one is based on Khasminskii's test but, in what follows, we will not deal with it; we refer the interested reader to \cite{bmpr21} and \cite{amr16}.The second condition reads
\begin{equation}
\label{I1.7}
\liminf_{R\to+\infty}\frac{1}{R^2}\log\vol B_R<+\infty
\end{equation}
and the third one involves the positive part $u_{+}=\max\{u,0\}$ of $u$ via the request
\begin{equation}
\label{I1.8}
\liminf_{R\to+\infty}\frac{1}{R^2}\log\int_{B_R}u_+^q<+\infty
\end{equation}
for some $q>0$.

The aim of this paper is to present conditions similar to (\ref{I1.7}) and (\ref{I1.8}) to guarantee that solutions $u$ of differential inequalities of the form
\[
	Lu \geq f(u)
\]
satisfy $f(u)\leq 0$ on $M$, without requiring the validity of (\ref{I1.5}), when $L$ is a weakly 1-coercive operator as defined in section \ref{sec2} below (for instance, the mean curvature operator falls in this class). The case of weakly $p$-coercive operators with $p>1$, including the $p$-Laplacian and in particular the Laplace-Beltrami operator for $p=2$, has been considered in the recent paper \cite{acr24}.

In what follows we mainly present two kind of results, see for instance Theorem \ref{A_1w_base} and Theorem \ref{Xlessreg}. In the first case, the proof has a geometric flavour in the vein of the paper \cite{acdn21} and requires more regularity; in the second case we refine our approach to make it less demanding from the point of view of regularity. As a consequence of the main results we infer comparison principles. Furthermore we apply them to geometric situations dealing with the mean curvature operator, which is a typical weakly 1-coercive operator. Comparison principles and related uniqueness results for the mean curvature operator have been considered by many authors especially in the Euclidean setting. Just to mention some references we refer the reader to Concus and Finn \cite{cf74a,cf74b}, Finn \cite{f53,f74}, Finn and Hwang \cite{fh89} and Hwang \cite{h88} for results on the Euclidean space $\R^m$, while for results on general Riemannian manifolds we recall Pigola, Rigoli and Setti \cite{prs02}.

\section{Weakly $1$-coercive operators}
\label{sec2}
	
Let $(M,\metric)$ be a complete, non-compact, Riemannian manifold of dimension $m\geq 2$. Given a non-empty open set $\Omega\subseteq M$, we denote by $TM|_\Omega = \cup_{x\in\Omega} T_x M$ the restriction of the tangent bundle of $M$ to $\Omega$ and by $W^{1,1}_\loc(\Omega)$ the space of measurable  functions $u : \Omega \to \RR$ such that for any relatively compact open subset $\Omega_0\subseteq\Omega$ the restriction $u|_{\Omega_0}$ belongs to the Sobolev space $W^{1,1}(\Omega_0)$, characterized as the completion of $C^1(\Omega_0)$ with respect to the Sobolev norm
\[
	\|u\|_{W^{1,1}(\Omega_0)} = \int_{\Omega_0} (|u| + |\nabla u|) \qquad \forall \, u \in C^1(\Omega_0) \, .
\]
In this paper, we consider quasilinear differential operators $L$ in divergence form weakly defined on functions $u\in W^{1,1}_\loc(\Omega)$ by
\begin{equation} \label{L_def}
	Lu = \div (A(x,u,\nabla u))
\end{equation}
or, more generally, by
\begin{equation} \label{Lw_def}
	Lu = w^{-1} \div (w \, A(x,u,\nabla u)) \, .
\end{equation}
Here $A : \RR\times TM|_\Omega \to TM|_\Omega$ is a \emph{Carath\'eodory-type (generalized) bundle map}, that is, a function such that
\[
	A(x,s,\xi) \in T_x M \qquad \forall \, x\in\Omega, \, s\in\RR, \, \xi\in T_x M
\]
and whose representation $\tilde A : \psi(U) \times \RR \times \RR^m \to \RR^m$ in any chart $\psi : U \to \RR^m$ (with $U\subseteq \Omega$) satisfies the Carath\'eodory conditions, that is,
\begin{itemize}
	\item $\tilde A(y,\,\cdot\,,\,\cdot\,) : \RR\times\RR^m \to \RR^m$ is continuous for a.e.~$y\in\psi(U)$, and
	\item $\tilde A(\,\cdot\,,s,v) : \psi(U) \to \RR^m$ is measurable for every $(s,v)\in\RR\times\RR^m$.
\end{itemize}
Throughout the paper, we assume $A$ to be \emph{weakly-1-coercive}, that is, we require that there exists a positive constant $k>0$ such that
\begin{align}
	\label{wc1}
	\langle A(x,s,\xi),\xi \rangle \geq 0 & \qquad \forall \, x \in\Omega, \, s \in \RR, \, \xi \in T_x M \\
	\label{wc2}
	A(x,s,0) = 0 & \qquad \forall \, x \in\Omega, \, s \in \RR \\
	\label{wc3}
	|A(x,s,\xi)| \leq k & \qquad \forall \, x \in\Omega, \, s \in \RR, \, \xi \in T_x M \, .
\end{align}
A Carath\'eodory-type generalized bundle map $A$ satisfiying only \eqref{wc1} and \eqref{wc2} is usually said to be \emph{weakly elliptic}. We adopt the said terminology for bundle maps $A$ satisfying all three conditions above because we regard \eqref{wc3} as the limit case for $p\to 1^+$ of the so-called weakly-$p$-coercivity (with $1 < p < +\infty$) condition for weakly elliptic bundle maps,
\[
	|A(x,s,\xi)| \leq k \langle A(x,s,\xi),\xi \rangle^{\frac{p-1}{p}} \, .
\]

If $A$ is weakly-$1$-coercive, then the differential operator $L$ associated to $A$ weakly defined by \eqref{L_def} is also said to be weakly-$1$-coercive, and the constant $k>0$ appearing in \eqref{wc3} is said to be a \emph{coercivity constant} for $L$. To be precise, $L$ is defined as the map
\[
	L : W^{1,1}_\loc(\Omega) \to W^{1,1}_c(\Omega)^\ast
\]
sending each $u\in W^{1,1}_\loc(\Omega)$ to the continuous linear functional on $W^{1,1}_c(\Omega)$ given by
\[
	W^{1,1}_c(\Omega) \ni \varphi \mapsto - \int_{\Omega} \langle A(x,u,\nabla u),\nabla\varphi\rangle \, .
\]
This is well defined because the structural conditions on $A$ ensure that for every $u\in W^{1,1}_\loc(\Omega)$ the function $\Omega \ni x \mapsto A(x,u(x),\nabla u(x))$ is a bounded measurable vector field on $\Omega$. For $u\in W^{1,1}_\loc(\Omega)$ and $f\in L^\infty_\loc(\Omega)$, we write
\[
	Lu = f \qquad \text{(weakly) in } \, \Omega
\]
if the functional $Lu$ can be represented as integration against $f$, that is, if
\begin{equation} \label{Lu_weakdef}
	- \int_\Omega \langle A(x,u,\nabla u),\nabla\varphi\rangle = \int_\Omega f\varphi \qquad \forall \, \varphi \in W^{1,1}_c(\Omega) \, .
\end{equation}
More generally, in case $f : \Omega \to \RR$ is a measurable function satisfying $f_+\in L^\infty_\loc(\Omega)$, then we write
\[
	Lu \geq f \qquad \text{(weakly) in } \, \Omega
\]
for a given function $u\in W^{1,1}_\loc(\Omega)$ if
\[
	-\int_{\Omega} \langle A(x,u,\nabla u),\nabla\varphi \rangle \geq \int_\Omega f \varphi \qquad \forall \, \varphi \in W^{1,1}_c(\Omega) , \, \varphi \geq 0 \, .
\]
The meaning of $Lu \leq f$, with $f : \Omega \to \RR$ measurable and such that $f_-\in L^\infty_\loc(\Omega)$, is defined analogously.

Of course, in case $A : \RR\times TM|_\Omega \to TM|_\Omega$ is $\mathcal C^1$ regular and $u\in\mathcal C^2(\Omega)$, then the weak definition of $Lu$ given above agrees with the strong one, that is, $A(x,u,\nabla u)$ is a $\mathcal C^1$ vector field and its divergence $f = \div(A(x,u,\nabla u))$ is pointwise well-defined and continuous in $\Omega$, and satisfies \eqref{Lu_weakdef}.

If together with a weakly-$1$-coercive Carath\'eodory bundle map $A$ we also consider a function $w\in L^\infty_\loc(\Omega)$ satisfying $w > 0$ a.e.~in $\Omega$, then the ``\emph{weighted}'' operator $L$ weakly defined by \eqref{Lw_def}, that is, implicitely determined by
\[
	w Lu = \div(w \, A(x,u,\nabla u)) \qquad \forall \, u \in W^{1,1}_\loc(\Omega) \, ,
\] will still be called a weakly-$1$-coercive operator (with coercivity constant $k>0$). Similarly to what we did above, if $u\in W^{1,1}_\loc(\Omega)$ and $f\in L^\infty_\loc(\Omega)$ then we write
\[
	Lu = f \qquad \text{(weakly) in } \, \Omega
\]
for $L$ as above if
\begin{equation} \label{Lwu_weakdef}
	- \int_\Omega w \langle A(x,u,\nabla u),\nabla\varphi\rangle = \int_\Omega wf\varphi \qquad \forall \, \varphi \in W^{1,1}_c(\Omega) \, ,
\end{equation}
and in case $f : \Omega \to \RR$ is measurable and such that $f_+\in L^\infty_\loc(\Omega)$ then we also write
\[
	Lu \geq f \qquad \text{(weakly) in } \, \Omega
\]
if
\[
	-\int_{\Omega} w \langle A(x,u,\nabla u),\nabla\varphi \rangle \geq \int_\Omega w f \varphi \qquad \forall \, \varphi \in W^{1,1}_c(\Omega) , \, \varphi \geq 0 \, ,
\]
and a similar definition can be given for the reversed inequality $Lu \leq f$ in case $f_-\in L^\infty_\loc(\Omega)$. Again, in case $A : \RR\times TM|_\Omega \to TM|_\Omega$ is $\mathcal C^1$ regular and $u\in\mathcal C^2(\Omega)$, if we also have $w\in\mathcal C^1(\Omega)$ then the weak definition of $Lu$ coincides with the strong one, meaning that the ``weighted'' divergence
\[
	f = w^{-1} \div(w \, A(x,u,\nabla u)) \equiv \div(A(x,u,\nabla u)) + \frac{\langle\nabla w,A(x,u,\nabla u)\rangle}{w}
\]
is a pointwise well-defined and continuous function in $\Omega$ satisfying \eqref{Lwu_weakdef}.

\section{Maximum principles for weakly-$1$-coercive operators}

Let $(M,\metric)$ be a complete Riemannian manifold of dimension $m\geq 2$. Throughout this section $L$ will always be a (possibly weighted) weakly-$1$-coercive operator associated to a $\mathcal C^1$ bundle map $A$ and acting on functions $u$ of class $\mathcal C^2$. In particular, $A(x,u,\nabla u)$ will always be a $\mathcal C^1$ vector field and all differential equalities and inequalities can be interpreted in a strong, pointwise sense.

In order to state our first result, we have to introduce two definitions. Let $o\in M$ be a fixed point. We denote with $r$ the Riemannian distance function from $o$ in $M$, that is, we write $r(x) = \dist(o,x)$ for $x\in M$, and we denote with $B_R = B_R(o)$ the geodesic ball of radius $R>0$ centered at $o$. Let $\mu\in [0,1]$ be a given a parameter. First, we define
\begin{equation} \label{cmu_M}
	c_\mu(M,o) = \begin{cases}
		\displaystyle\liminf_{R\to+\infty} \frac{(1-\mu) \log\vol(B_R)}{R^{1-\mu}} & \quad \text{if } \, 0 \leq \mu < 1 \\[0.4cm]
		\displaystyle\liminf_{R\to+\infty} \frac{\log\vol(B_R)}{\log R} & \quad \text{if } \, \mu = 1 \, .
	\end{cases}
\end{equation}
As a consequence of the triangle inequality, the value of $c_\mu(M,o)$ is actually independent on the choice of the base point $o\in M$, so we will simply write $c_\mu(M)$ for $c_\mu(M,o)$. Moreover, we always have
\[
	c_\mu(M) \geq 0
\]
because, $\log\vol(B_R)$ being a nondecreasing function of $R>0$, we have
\[
c_\mu(M) \geq \begin{cases}
	\displaystyle\liminf_{R\to+\infty} \frac{(1-\mu) \log\vol(B_1)}{R^{1-\mu}} = 0 & \quad \text{if } \, 0 \leq \mu < 1 \\[0.4cm]
	\displaystyle\liminf_{R\to+\infty} \frac{\log\vol(B_1)}{\log R} = 0 & \quad \text{if } \, \mu = 1 \, .
\end{cases}
\]
We will comment further below on exact values and lower bounds on $c_\mu(M)$ for some particular classes of manifolds. Secondly, we will say that a positive continuous function $b : M \to (0,+\infty)$ \emph{satisfies condition} \eqref{condB} for a couple of given parameters $\mu\in[0,1]$ and $\beta>0$ when the following circumstance is verified:
\begin{equation} \label{condB} \tag{B}
	\left\{
		\begin{array}{r@{\;}c@{\;}l@{\qquad}l}
			\displaystyle\liminf_{r(x)\to+\infty} r(x)^\mu b(x) & \geq & \beta & \text{in case } \, 0 < \mu \leq 1 \\
			b(x) & \geq & \beta \quad \forall \, x \in M & \text{in case } \, \mu = 0 \, .
		\end{array}
	\right.
\end{equation}
Again, for a given continuous positive function $b(x)$ on $M$ the validity of condition (B) for a given pair of parameters $\mu\in[0,1]$ and $\beta>0$ does not depend on the choice of the base point $o\in M$ with respect to which the distance $r(x)$ is measured. In case $\mu\in(0,1]$, condition (B) amounts to saying that for each $\beta^\ast < \beta$ there exists $R_0 = R_0(\beta^\ast) > 0$ such that
\[
	b(x) \geq \beta^\ast r(x)^{-\mu} \qquad \text{for } \, r(x) > R_0 \, ,
\]
that is, $b(x)$ does not decay faster than $\beta r(x)^{-\mu}$ as $x$ goes to infinity in $M$ along any diverging sequence. In particular, in case $\mu\in(0,1]$ condition \eqref{condB} is satisfied, for a given $\beta>0$, by any continuous $b : M \to (0,+\infty)$ such that
\[
	b(x) \geq \beta r(x)^{-\mu} \qquad \text{for } \, r(x) \gg 1 \, .
\]

Our first result in this setting is the following
\begin{theorem} \label{A_1w_base}
	Let $(M,\metric)$ be a complete manifold, $A : \RR \times TM \to TM$ a $\mathcal{C}^1$ weakly-$1$-coercive generalized bundle map with coercivity constant $k>0$ and let $L$ be the corresponding weakly-$1$-coercive operator given by \eqref{L_def}.
	
	Let $f\in\mathcal C(\RR)$ be non-decreasing and $b\in\mathcal C(M)$ be positive and satisfy condition \eqref{condB} for parameters $\mu\in[0,1]$ and $\beta>0$. If $u\in\mathcal C^2(M)$ satisfies
	\begin{equation}
		\label{L1_Om_M}
		Lu \geq b(x)f(u) \qquad \text{on } \, M, 
	\end{equation}
	then
	\begin{equation}\label{f_cmu_M}
	f(u) \leq \frac{k}{\beta} c_\mu(M)
	\qquad \text{on } \, M .
	\end{equation}
	In particular, if $c_\mu(M) = 0$ then
	\[
		f(u) \leq 0 \qquad \text{on } \, M .
	\]
\end{theorem}

\begin{remark}\label{remark_cmu}
	\emph{Regarding the values of $c_\mu(M)$ a few remarks are in order, in particular to highlight some relevant cases in which $c_\mu(M) = 0$. If $M=\RR^m$ is the Euclidean space then $\vol(B_R)=\omega_mR^m$, where $\omega_m$ denotes the volume of the Euclidean unit ball of the same dimension. As a consequence
	\begin{equation} \label{cmuRm}
	c_\mu(\RR^m)=\begin{cases}
	0 & \quad \text{if } \, 0 \leq \mu < 1 \\
	m & \quad \text{if } \, \mu = 1 \, .
	\end{cases}
	\end{equation}
	More generally, if $(M,\metric)$ has at most polynomial volume growth, that is,
	\[
	\limsup_{R\to+\infty} \frac{\vol(B_R)}{R^N} < +\infty
	\]
	for some real number $N>0$, then
	\[
	c_\mu(M)\begin{cases}
	= 0 & \quad \text{if } \, 0 \leq \mu < 1 \\
	\leq N & \quad \text{if } \, \mu = 1 \, .
	\end{cases}
	\]
	Even more, if $M$ has finite volume, then $\vol(B_R)\leq\vol(M)<+\infty$ for every $R>0$ and $c_\mu(M)= 0$ for every $\mu\in[0,1]$. If $M=\mathbb{H}^m$ is the hyperbolic space of constant sectional curvature $-1$ then it is known that $\vol(B_R)\sim c_me^{(m-1)R}$ as $R\rightarrow+\infty$, where $c_m=\omega_{m-1}2^{1-m}$, and then
	\begin{equation} \label{cmuHm}
	c_\mu(\mathbb{H}^m)=\begin{cases}
	m-1 & \quad \text{if } \, \mu=0 \\
	+\infty & \quad \text{if } \, 0<\mu\leq 1 \, .
	\end{cases}
	\end{equation}
	More generally, by Bishop-Gromov volume comparison theorem, if $(M,\metric)$ is complete and
	\begin{equation} \label{RicB}
	\Ricc(\nabla r,\nabla r) \geq - (m-1) \frac{B^2}{(1+r^2)^{\alpha/2}} \qquad \text{with } \, 0 < \alpha \leq 2 \, , B \geq 0
	\end{equation}
	then we can estimate $c_\mu(M)$ from above in cases where $0\leq\mu\leq\frac{\alpha}{2}$. Using Theorem 3.11 in \cite{bmpr21} (and also taking into account the subsequent discussion there) we have
	\begin{itemize}
		\item if $0 < \alpha < 2$ then 
		\[
		c_\mu(M) \begin{cases}
		= 0 & \quad \text{if } \, 0 \leq \mu < \dfrac{\alpha}{2} \\[0.2cm]
		\leq (m-1)B & \quad \text{if } \, \mu = \dfrac{\alpha}{2} \, .
		\end{cases}
		\]
		\item if $\alpha=2$ then
		\[
		c_\mu(M) \begin{cases}
		= 0 & \quad \text{if } \, 0 \leq \mu < 1 \\
		\leq (m-1)\dfrac{1+\sqrt{1+4B^2}}{2} + 1 & \quad \text{if } \, \mu = 1 \, .
		\end{cases}
		\]
	\end{itemize}
	Also, if \eqref{RicB} holds with $\alpha=0$ and $B\geq0$ we still can estimate $c_0(M) \leq (m-1)B$ for $\mu=0$. Note that these estimates are sharp in view of \eqref{cmuRm} and \eqref{cmuHm} for $\mu=0$.}
\end{remark}

\begin{proof}[Proof of Theorem \ref{A_1w_base}]
	Set $c^\ast=\frac{k}{\beta}c_\mu(M)$ and suppose, by contradiction, that \eqref{f_cmu_M} is false. Choose $\eps>0$ small enough so that
	$$
	\Omega_\eps = \{ x \in M : f(u(x)) > c^\ast + \eps \} \neq \emptyset \, .
	$$
	Since $f$ is continuous and monotone nondecreasing, there exists $s_\eps\in\RR\cup\{-\infty\}$ such that
	$$
	\Omega_\eps = \{ x \in M : u(x) > s_\eps \}
	$$
	and $\Omega_\eps$ is open by continuity of $u$. 
	Consider the vector field $X:M \to TM$ given by
	$$
	X = A(x,u(x),\nabla u(x)) \, .
	$$
	The assumptions on $A$ and $u$ ensure that $X$ is of class $\mathcal C^1$ and that $|X|\leq k$ on $M$. Therefore, since $M$ is complete, we know that $X$ has a flow $\Psi(t,x)$ well-defined for every $t\in\RR$ and $x\in M$.  We observe now that $\Psi:[0,+\infty) \times \Omega_\eps \to \Omega_\eps$. Actually, for any $x\in\Omega_\eps$ denote $\gamma(t)=\Psi(t,x)$. Then 
	$$
	\frac{\di}{\di t} u(\gamma(t)) = \langle \nabla u, \dot\gamma(t) \rangle_{\gamma(t)} = \langle \nabla u, X_{\gamma(t)} \rangle_{\gamma(t)} = \langle \nabla u, A(\gamma(t),u,\nabla u) \rangle_{\gamma(t)} \geq 0 \, ,
	$$
	which implies $u(\gamma(t)) \geq u(\gamma(0)) = u(x) > s_\eps$ for every $t\in[0,+\infty)$. Thus $\gamma(t)$ lies in $\Omega_\eps$ for every $t\in[0,+\infty)$.
	
	Now fix $p\in\Omega_\eps$ and $\delta>0$ such that $B_\delta(p) \subseteq \Omega_\eps$ and $\partial B_\delta(p)$ is smooth. For every $t\in[0,+\infty)$ let
	$$
	U_t = \Psi(t,B_{\delta}(p))=\Psi_t(B_{\delta}(p))
	$$
	and 
	$$
	g(t)=\vol(U_t)=\int_{\Psi_t(B_{\delta}(p))}\di v=\int_{B_{\delta}(p)}\Psi_t^{*}(\di v).
	$$
	Since $\overline{B_{\delta}(p)}$ is compact, we can differentiate under the integral and, using that $\mathcal L_X(\di v) = \div X \, \di v$, where $\mathcal L_X$ denotes the Lie derivative with respect to $X$, we get
	\begin{eqnarray*}
		g'(t) & = & \left.\frac{\di}{\di\tau}\right|_{\tau=0}\int_{\Psi_{\tau+t}(B_{\delta}(p))}\di v= 
		\left.\frac{\di}{\di\tau}\right|_{\tau=0}\int_{\Psi_{t}(B_{\delta}(p))}\Psi_\tau^*(\di v) \\ [0.2cm]
		{} & = & \int_{\Psi_{t}(B_{\delta}(p))}\left.\frac{\di}{\di\tau}\right|_{\tau=0}\Psi_\tau^*(\di v)= 	\int_{U_t}\div X\di v.
	\end{eqnarray*}
	Since $\Psi(t,\Omega_\eps)\subseteq\Omega_\eps$, we have $U_t\subseteq\Omega_\eps$ and, by  \eqref{L1_Om_M}, we have
	\begin{equation} \label{g'_M}
	g'(t) = \int_{U_t} \div X\di v=\int_{U_t}Lu\di v \geq \int_{U_t} b(x) f(u)\di v > (c^\ast+\eps) \int_{U_t} b(x)\di v \, .
	\end{equation}
	By continuity, we can choose $\beta^\ast\in\RR$ satisfying
	\[
		\beta > \beta^\ast > \frac{c^\ast}{c^\ast+\eps} \beta \, .
	\]
	Note that, by the definition of $c^\ast$, the second inequality amounts to
	\begin{equation} \label{betast}
		\frac{(c^\ast+\eps)\beta^\ast}{k} > c_\mu (M) \, .
	\end{equation}
	Because of the assumptions on $b$ and since $\beta^\ast < \beta$, we can choose $T>0$ such that
	\begin{equation}
	b(x) \geq \beta^\ast r(x)^{-\mu} \qquad \forall \, x\not\in \overline{B_T} \, .
	\end{equation}
	Then, there exists $R_0>0$ such that
	\begin{equation} \label{inf_b_M}
	\inf_{B_R} b(x) \geq \beta^\ast R^{-\mu} \qquad \forall \, R\geq R_0 \, .
	\end{equation}
	Indeed, it suffices to choose $R_0>T$ such that $\beta^\ast R_0^{-\mu} \leq \min_{\overline{B_T}}b>0$. .In that case and for every $R>R_0$, if $x\in B_R\setminus \overline{B_T}$ then we know that $b(x)\geq \beta^\ast r(x)^{-\mu}\geq \beta^\ast R^{-\mu}$. On the other hand, if $x\in\overline{B_T}$ then $b(x)\geq \min_{\overline{B_T}}b\geq \beta^\ast R_0^{-\mu}\geq \beta^\ast R^{-\mu}$. Summing up, for every $R>R_0$ and for every $x\in B_R$ we have $b(x)\geq \beta^\ast R^{-\mu}$, that is \eqref{inf_b_M}.
	
	As a consequence of $|X|\leq k$ and of the triangle inequality,
	\begin{equation}
	\label{luis1_M}
	d(p,x)<\delta+kt \qquad \forall x\in U_t.
	\end{equation}
	To see it, let $x=\Psi(t,y)\in U_t$ with $y\in B_\delta(p)$. Then 
	\[
	d(p,x)=d(p,\Psi(t,y))\leq d(p,y)+d(y,\Psi(t,y))<\delta+d(y,\Psi(t,y)),
	\]
	with
	\[
	d(y,\Psi(t,y))\leq\int_0^t\left|\frac{\partial\Psi}{\partial s}(s,y)\right|\di s=
	\int_0^t|X_{\Psi(s,y)}|\di s\leq kt,
	\]
	which implies \eqref{luis1_M}. Therefore
	\begin{equation} \label{Ut_in_M}
	U_t \subseteq B_{\delta+kt}(p) \subseteq B_{r(p)+\delta+kt} \qquad \forall \, t \geq 0 \, .
	\end{equation}
	From \eqref{g'_M} we have
	$$
	g'(t) \geq (c^\ast+\eps)\, \inf_{U_t}b(x) \, g(t) \qquad \forall \, t \geq 0 \, .
	$$
	Setting $R_1 = \max\{R_0,r(p)+\delta\}$, from \eqref{inf_b_M} and \eqref{Ut_in_M} we get 
	$$
	\inf_{U_t}b(x)\geq \inf_{B_{r(p)+\delta+kt}}b(x)\geq (R_1+kt)^{-\mu}\beta^\ast,
	$$
	and then 
	$$
	g'(t) \geq (c^\ast+\eps)\beta^\ast (R_1+kt)^{-\mu} g(t) \qquad \forall \, t \geq 0 \, ,
	$$
	equivalently,
	\[
	\frac{g'(t)}{g(t)} \geq (c^\ast+\eps)\beta^\ast (R_1+kt)^{-\mu}  \qquad \forall \, t \geq 0 \, .
	\]
	Integrating this inequality we obtain
	\begin{equation} \label{g'2_M}
	\log (g(t)) \geq \log(g(0)) + (c^\ast+\eps)\beta^\ast \int_0^t (R_1+ks)^{-\mu} \, \di s \, .
	\end{equation}
	Since $U_t \subseteq B_{\delta+kt}(p) \subseteq B_{R_1+kt}$, we have $g(t) \leq \vol(B_{R_1+kt})$, thus letting $R=R_1+kt$ and changing the integration variable in \eqref{g'2_M} we get
	\begin{eqnarray*}
		\log \vol(B_R) & \geq & \log(g(t)) \\[0.3cm]
		{} & \geq & \log(g(0)) + \frac{(c^\ast+\eps)\beta^\ast}{k} \int_{R_1}^R s^{-\mu} \, \di s\\ [0.3cm]
		{} & = & \log(g(0))+\frac{(c^\ast+\eps)\beta^\ast}{k}\begin{cases}
			\displaystyle \left(\frac{R^{1-\mu}}{1-\mu}-\frac{R_1^{1-\mu}}{1-\mu}\right) & \quad \text{if } \, 0 \leq \mu < 1 \\[0.4cm]
			\displaystyle \left(\log R-\log R_1\right) & \quad \text{if } \, \mu = 1
		\end{cases}
	\end{eqnarray*}
	for every $R\geq R_1$. In other words, for all $R\geq R_1$ we respectively have
	\[
		\frac{(1-\mu)\log\vol(B_R)}{R^{1-\mu}} \geq \frac{(c^\ast+\eps)\beta^\ast}{k} \left(1-\frac{R_1^{1-\mu}}{R^{1-\mu}}\right) + \frac{(1-\mu)\log(g(0))}{R^{1-\mu}}
	\]
	in case $0\leq \mu < 1$, or
	\[
		\frac{\log\vol(B_R)}{\log R} \geq \frac{(c^\ast+\eps)\beta^\ast}{k} \left(1-\frac{\log R_1}{\log R}\right) + \frac{\log(g(0))}{\log R}
	\]
	in case $\mu=1$. In both cases, letting $R\to+\infty$ and applying the definition \eqref{cmu_M} of $c_\mu(M)$ we conclude from here that
	\[
		c_\mu(M) \geq \frac{(c^\ast+\eps)\beta^\ast}{k} \, ,
	\]
	contradicting \eqref{betast}.
\end{proof}

We can ``localize'' Theorem \ref{A_1w_base} as follows: if $\Omega\subsetneqq M$ is a non-empty open set and $\mu\in[0,1]$ is a given parameter, then fixing $o\in \Omega$ we define
\begin{equation} \label{cmu_logv}
	c_\mu(\Omega) = \begin{cases}
		\displaystyle\liminf_{R\to+\infty} \frac{(1-\mu) \log\vol(B_R\cap\Omega)}{R^{1-\mu}} & \quad \text{if } \, 0 \leq \mu < 1 \\[0.4cm]
		\displaystyle\liminf_{R\to+\infty} \frac{\log\vol(B_R\cap\Omega)}{\log R} & \quad \text{if } \, \mu = 1
	\end{cases}
\end{equation}
where $B_R = B_R(o)$. As in case $\Omega = M$, the resulting value of $c_\mu(\Omega)$ is independent of the choice of the base point $o\in \Omega$. Moreover, the set function $c_\mu$ is non-negative and monotonic non-decreasing with respect to inclusion, that is, for every $\mu\in[0,1]$ and for any pair of non-empty open sets $\Omega_1,\Omega_2\subseteq M$ with $\Omega_1\subseteq\Omega_2$ we have
\[
	0 \leq c_\mu(\Omega_1) \leq c_\mu(\Omega_2) \leq c_\mu(M) \, .
\]

\begin{theorem} \label{A_1w_base_bis}
	Let $(M,\metric)$ be a complete manifold, $\Omega\subsetneqq M$ a non-empty open set, $A : \RR \times TM|_\Omega \to TM|_\Omega$ a $\mathcal{C}^1$ weakly-$1$-coercive generalized bundle map with coercivity constant $k>0$ and let $L$ be the corresponding operator given by \eqref{L_def}. Let $f\in\mathcal C(\RR)$ be non-decreasing and $b\in\mathcal C(M)$ be positive and satisfy condition \eqref{condB} for given parameters $\mu\in[0,1]$ and $\beta>0$. If $u\in\mathcal C(\overline{\Omega}) \cap \mathcal C^2(\Omega)$ satisfies
	\begin{equation} \label{L1_Om}
		Lu \geq b(x) f(u) \qquad \text{on } \, \Omega \, ,
	\end{equation}
	then
	\begin{equation} \label{f_cmu}
		f(u) \leq \max\left\{\frac{k}{\beta}c_\mu(\Omega),\sup_{\partial\Omega} f(u)\right\} \quad \text{on } \, \Omega \, .
	\end{equation}
	In particular, if $c_\mu(\Omega) = 0$ then
	\begin{equation} \label{f_cmu_bis}
		f(u) \leq \max\left\{0,\sup_{\partial\Omega} f(u)\right\} \quad \text{on } \, \Omega.
	\end{equation}
\end{theorem}

\begin{proof}[Proof of Theorem \ref{A_1w_base_bis}]
	If $\sup_{\partial\Omega} f(u)=+\infty$ then the conclusion is trivially satisfied; otherwise we proceed in a way similar to that of Theorem \ref{A_1w_base} and set $c^\ast =\max\left\{\frac{k}{\beta}c_\mu(\Omega),c\right\}$, with $c=\sup_{\partial\Omega} f(u)\in\R$. Note that $c^\ast \geq 0$, since $c_\mu(\Omega) \geq 0$ as remarked above. Suppose, by contradiction, that \eqref{f_cmu} is false and choose $\eps>0$ small enough so that
	\[
		\Omega_\eps = \{ x \in \Omega : f(u(x)) > c^\ast + \eps \} \neq \emptyset \, ,
	\]
	which is open by continuity of $f$ and $u$. Let us see that $\overline{\Omega_\eps}\subseteq\Omega$. Indeed, let $\tilde x\in\overline{\Omega_\eps}$, then, for some $\{x_j\}_{j\in\N} \subset \Omega_\eps$, $\tilde x = \lim x_j$ and then $f(u(\tilde x)) = \lim f(u(x_j)) \geq c^\ast+\eps > c$, which means that $\tilde x\not\in\partial\Omega$ and hence $\tilde x\in\Omega$. Therefore, there exists $s_\eps\in\RR$ such that
	\[
		\Omega_\eps = \{ x \in \Omega : u(x) > s_\eps \} \, .
	\]
	
	Consider the vector field $X:\Omega \to TM$ defined on $\Omega\subsetneqq M$ by
	\[
		X = A(x,u(x),\nabla u(x)) \, .
	\]
	The assumptions on $A$ and $u$ ensure that $X$ is of class $\mathcal C^1(\Omega)$ and that $|X|\leq k$ on $\Omega$. Although $X$ is not defined on all of $M$, we observe that $X$ still has a well-defined flow $\Psi : [0,+\infty) \times \Omega_\eps \to \Omega_\eps$. Indeed, we know that for any $x\in\Omega_\eps$ there exists a maximal time $T_x \in (0,+\infty]$ such that the initial value problem
	\begin{equation} \label{CpX}
	\begin{cases}
	\gamma(0) = x\in\Omega_\eps \\
	\dot\gamma(t) = X_{\gamma(t)}
	\end{cases}
	\end{equation}
	has a (unique) solution $\gamma : [0,T_x) \to \Omega$. We will show that $T_x=+\infty$. Towards this aim, we observe that since
	\[
		\frac{\di}{\di t} u(\gamma(t)) = \langle \nabla u, \dot\gamma(t) \rangle_{\gamma(t)} = \langle \nabla u, X_{\gamma(t)} \rangle_{\gamma(t)} = \langle \nabla u, A(\gamma(t),u,\nabla u) \rangle_{\gamma(t)} \geq 0 \, ,
	\]
	we have $u(\gamma(t)) \geq u(\gamma(0)) = u(x) > s_\eps$ for every $t\in[0,T_x)$. Thus $\gamma(t)$ lies in $\Omega_\eps$ for every $t\in[0,T_x)$. Suppose by contradiction that $T_x < +\infty$. Since $|X| \leq k$, the curve $\gamma$ has length $\ell(\gamma) \leq kT_x < +\infty$. As $M$ is complete, $\gamma$ has a limit endpoint $\tilde x = \lim_{t\to T_x} \gamma(t) \in M$. Since $\gamma([0,T_x)) \subseteq \Omega_\eps$ we have $\tilde x \in \overline{\Omega_\eps} \subseteq \Omega$,  Thus $u(\tilde x)$ is defined and by continuity 
	$u(\tilde x) = \lim_{t\to T_x} u(\gamma(t))\geq u(x) > s_\eps$. But then $\tilde x \in \Omega_\eps$ and thus $T_x$ is not the maximal existence time for the solution of \eqref{CpX}, contradiction. That is, $T_x = +\infty$.
	
	Once we have guaranteed that $\Psi(t,\Omega_\eps)\subseteq\Omega_\eps$ for every $t\geq 0$, the proof follows exactly the same lines as the proof of Theorem \ref{A_1w_base} (including the choice of $\beta^\ast\in\RR$ satisfying $\beta > \beta^\ast > \frac{c^\ast}{c^\ast+\eps}\beta$) until inequality \eqref{g'2_M}. From that point on, since $U_t \subseteq B_{\delta+kt}(p) \subseteq B_{R_1+kt}$ and also $U_t\subseteq\Omega_\eps\subseteq\Omega$, we have $g(t) \leq \vol(B_{R_1+kt}\cap\Omega)$, and proceeding as in the proof of Theorem \ref{A_1w_base} we obtain
	\[
	\log \vol(B_R\cap\Omega) \geq \log(g(0))+\frac{(c^\ast+\eps)\beta^\ast}{k}
	\begin{cases}
	\displaystyle \left(\frac{R^{1-\mu}}{1-\mu}-\frac{R_1^{1-\mu}}{1-\mu}\right) & \quad \text{if } \, 0 \leq \mu < 1 \\[0.4cm]
	\displaystyle \left(\log R-\log R_1\right) & \quad \text{if } \, \mu = 1 \, .
	\end{cases}
	\]
	for every $R\geq R_1$. Hence, by \eqref{cmu_logv} we conclude from here that
	$c_\mu(\Omega)\geq (c^\ast+\eps)\beta^\ast/k> c^\ast\beta/k\geq c_\mu(\Omega)$, which is a contradiction.
\end{proof}

In the next theorems, we state the weighted version of our previous results, where by this we mean that we are working with operators $L$ of the form
\begin{equation} \label{Lw_2def}
	Lu = w^{-1}\div(w\, A(x,u,\nabla u))
\end{equation}
where $A$ is $\mathcal{C}^1$ regular and $w$ is a strictly positive $\mathcal{C}^1$ function, acting as a ``weight''. Given $h\in\mathcal{C}^{1}(M)$, it is standard to denote by $\div_h$ the ``weighted'' or ``drift'' divergence operator defined by
\[
	\div_h X = \div X - \langle\nabla h,X\rangle
\]
for vector fields $X:M\rightarrow TM$, and by $\vol_h$ the ``weighted'' volume defined by
\[
	\vol_h(U)=\int_{U}e^{-h}\di v
\]
for measurable subsets $U\subseteq M$. With this notation, the definition \eqref{Lw_2def} restates as
\[
	Lu = \div_h A(x,u,\nabla u) \qquad \text{for } \, h = -\log w \, .
\]
The next two theorems are stated in terms of weighted versions of the parameters $c_\mu(M)$ and $c_\mu(\Omega)$: fixing as usual an origin $o\in M$ and denoting $B_R = B_R(o)$, if $\Omega\subseteq M$ is open and non-empty then for $w\in \mathcal{C}^1(\Omega)$ such that $w>0$ in $\Omega$ and for $\mu\in[0,1]$ we define
\begin{equation} \label{cmu_MW}
	c^w_\mu(\Omega) = \begin{cases}
		\displaystyle\liminf_{R\to+\infty} \frac{(1-\mu) \log \int_{B_R\cap\Omega} w}{R^{1-\mu}} & \quad \text{if } \, 0 \leq \mu < 1 \\[0.4cm]
		\displaystyle\liminf_{R\to+\infty} \frac{\log \int_{B_R\cap\Omega} w}{\log R} & \quad \text{if } \, \mu = 1 \, .
	\end{cases}
\end{equation}

\begin{theorem} \label{A_1w_base_W} 
	Let $(M,\metric)$ and $A$ be as in Theorem \ref{A_1w_base}, let $w\in\mathcal C^1(M)$ be strictly positive and let $L$ be the corresponding operator given by \eqref{Lw_2def}. Let $f\in\mathcal C(\RR)$ be non-decreasing and $b\in\mathcal C(M)$ be positive and satisfy condition \eqref{condB} for given parameters $\mu\in[0,1]$ and $\beta>0$. If $u\in\mathcal C^2(M)$ satisfies
	\begin{equation} \label{L1_Om_MW}
		Lu \geq b(x)f(u) \qquad \text{on } \, M
	\end{equation}
	then
	\begin{equation}\label{f_cmu_MW}
		f(u) \leq \frac{k}{\beta} c^w_\mu(M) \qquad \text{on } \, M \, .
	\end{equation}
	In particular, if $c^w_\mu(M)=0$ then
	\[
		f(u) \leq 0 \qquad \text{on } \, M .
	\]
\end{theorem}

\begin{proof}
	Setting $h=-\log w$, the proof goes exactly as that of Theorem \ref{A_1w_base}, after replacing the function $g(t)$ with the function
	\[
		g_h(t)=\vol_h(U_t)=\int_{\Psi_t(B_\delta(p))}e^{-h}\di v=\int_{B_\delta(p)}\Psi_t^*(e^{-h}\di v) \, ,
	\]
	by observing that
	\begin{eqnarray*}
		g_h'(t) & = & \left.\frac{\di}{\di\tau}\right|_{\tau=0} \int_{\Psi_{\tau+t}(B_\delta(p))}e^{-h}\di v=\left.\frac{\di}{\di\tau}\right|_{\tau=0}\int_{\Psi_{t}(B_\delta(p))}\Psi_\tau^*(e^{-h}\di v)\\
		{} & = & 
		\int_{\Psi_{t}(B_\delta(p))}\left.\frac{\di}{\di\tau}\right|_{\tau=0}\Psi_\tau^*(e^{-h}\di v)=\int_{\Psi_{t}(B_\delta(p))}\left.\frac{\di}{\di\tau}\right|_{\tau=0}\left(e^{-h\circ\Psi_\tau}\Psi_\tau^*(\di v)\right)\\
		{} & = & 
		\int_{\Psi_{t}(B_\delta(p))}e^{-h}(\div X-\langle\nabla h,X\rangle)\di v=\int_{\Psi_{t}(B_\delta(p))}e^{-h}\div_h X\di v.
	\end{eqnarray*}
\end{proof}

Similarly, reasoning as in the proof of Theorem \ref{A_1w_base_bis}, we also have

\begin{theorem} \label{A_1w_base_bisW}
	Let $(M,\metric)$, $\Omega$ and $A$ be as in Theorem \ref{A_1w_base_bis}, let $w\in\mathcal C^1(\Omega)$ be strictly positive and let $L$ be the corresponding operator given by \eqref{Lw_2def}. Let $f\in\mathcal C(\RR)$ be non-decreasing and $b\in\mathcal C(M)$ be positive and satisfy condition \eqref{condB} for given parameters $\mu\in[0,1]$ and $\beta>0$. If $u\in\mathcal C(\overline{\Omega}) \cap \mathcal C^2(\Omega)$ satisfies
	\begin{equation} \label{L1_OmW}
		L u \geq b(x) f(u) \quad \text{on } \, \Omega
	\end{equation}
	then
	\begin{equation} \label{f_cmuW}
		f(u) \leq \max\left\{\frac{k}{\beta}c^w_\mu(\Omega),\sup_{\partial\Omega}f(u)\right\} \quad \text{on } \, \Omega \, .
	\end{equation}
	In particular, if $c^w_\mu(\Omega)=0$ then 
	\begin{equation} \label{f_cmu_bisW}
		f(u) \leq \max\left\{0,\sup_{\partial\Omega}f(u)\right\} \quad \text{on } \, \Omega.
	\end{equation}
\end{theorem}

More generally, the same line of reasoning can be applied to first-order differential inequalities of the form $\div_h(X) \geq b(x) \psi$ with $X$ a bounded $\mathcal{C}^1$ vector field and $\psi$ a $\mathcal C^1$ function only related by the condition $\langle\nabla\psi,X\rangle \geq 0$.

\begin{theorem} \label{divX_1w}
	Let $(M,\metric)$ be a complete, non-compact manifold, $\Omega\subseteq M$ a non-empty open set and $X : \Omega \to TM$ a $\mathcal{C}^1$ vector field satisfying $|X|\leq k$ in $\Omega$ for some constant $k>0$. Let $b\in \mathcal{C}(M)$ be positive and satisfy condition \eqref{condB} for some $\mu\in[0,1]$ and $\beta>0$. Let $\psi\in \mathcal{C}^1(\Omega)\cap\mathcal C(\overline{\Omega})$ be such that
	\begin{equation}
		\langle \nabla \psi,X \rangle \geq 0 \qquad \text{on } \, \Omega
	\end{equation}
	and, for some $h\in \mathcal{C}^1(\Omega)$,
	\begin{equation}
		\div_h(X) \geq b(x) \, \psi  \quad \text{on } \, \Omega.
	\end{equation}
	Then
	\[
		\sup_\Omega \psi \leq \max\left\{\frac{k}{\beta} c^w_\mu(\Omega), \sup_{\partial\Omega}\psi\right\}
	\]
	where $w=e^{-h}$.
\end{theorem}

\section{Applications to capillary graphs and prescribed mean curvature graphs}

Here we introduce the mean curvature operator
\begin{equation} \label{MCO}
	\div\left(\frac{\nabla u}{\sqrt{1+|\nabla u|^2}}\right)
\end{equation}
as an example of weakly-1-coercive operator with coercivity constant $k=1$. 

In our next result we assume that $c_\mu(M)=0$ for some $\mu\in[0,1]$. Recall from Remark \ref{remark_cmu} that this is the case if $(M,\metric)$ has, for instance, at most polynomial volume growth and $\mu\in[0,1)$.

\begin{theorem}\label{thcapi1}
	Let $(M,\metric)$ be a complete noncompact Riemannian manifold with $c_\mu(M)=0$ for some $\mu\in[0,1]$ and let $0<b\in C(M)$ satisfy
	\[
	\left\{
	\begin{array}{l@{\quad}l@{\qquad}l}
		b(x) \geq \beta & \forall \, x \in M & \text{if } \, \mu = 0 \\[0.3cm]
		b(x) \geq \beta r(x)^{-\mu} & \text{for } \, r(x) \gg 1 & \text{if } \, 0 < \mu \leq 1
	\end{array}
	\right.
	\]
	for some $\beta>0$. Then the only globally defined $\mathcal{C}^2$ solution $u:M\rightarrow\R$ of the capillary surface equation 
	\begin{equation} \label{CSE}
		\div\left(\frac{\nabla u}{\sqrt{1+|\nabla u|^2}}\right) = b(x)u
	\end{equation}
	is the trivial solution $u\equiv 0$.
\end{theorem}

\begin{proof}
	Assume that $u\in\mathcal C^2(M)$ is a global solution of (\ref{CSE}). 
	Then, from the last assertion in Theorem \ref{A_1w_base} we deduce $u\leq 0$ on $M$. 
	Consider now the function $v=-u$, which clearly satisfies
	\[
	\div\left(\frac{\nabla v}{\sqrt{1+|\nabla v|^2}}\right)=b(x)v.
	\]
	Then we also obtain that $v\leq 0$ on $M$, in other words $u\geq 0$. Summing up $u\equiv 0$ on $M$.
\end{proof}

A ``localized'' version of the previous result holds as well.

\begin{theorem}\label{thcapi2}
	Let $(M,\metric)$ be a complete non-compact Riemannian manifold and let $\Omega\subsetneqq M$ be an open set. Let $0<b\in C(M)$ satisfy, for some parameters $\mu\in[0,1]$ and $\beta>0$,
		$$
		\left\{
		\begin{array}{l@{\quad}l@{\qquad}l}
			b(x) \geq \beta & \forall \, x \in M & \text{if } \, \mu = 0 \\[0.3cm]
			b(x) \geq \beta r(x)^{-\mu} & \text{for } \, r(x) \gg 1 & \text{if } \, 0 < \mu \leq 1 \, .
		\end{array}
		\right.
		$$
	Let $u\in\mathcal C(\overline{\Omega})\cap\mathcal C^2(\Omega)$ be a solution of the capillary surface equation on $\Omega$ 
	\begin{equation}
	\label{CSE_bis}
	\div\left(\frac{\nabla u}{\sqrt{1+|\nabla u|^2}}\right)=b(x)u.
	\end{equation}
	Then
	\[
		\sup_\Omega u \leq \max\left\{ \frac{c_\mu(\Omega)}{\beta}, \sup_{\partial\Omega}u \right\} \, , \qquad \inf_\Omega u \geq \min\left\{ -\frac{c_\mu(\Omega)}{\beta}, \inf_{\partial\Omega}u \right\}
	\]
	In particular, 
	\[
	\sup_\Omega|u|\leq \max\left\{\frac{c_\mu(\Omega)}{\beta},\sup_{\partial\Omega}|u|\right\}.
	\]
\end{theorem}

\begin{proof}
	Assume that $u\in\mathcal C(\overline{\Omega})\cap\mathcal C^2(\Omega)$ is a solution of the capillary surface equation on $\Omega$.
	Then, from Theorem \ref{A_1w_base_bis} we obtain that $u\leq\max\left\{\frac{c_\mu(\Omega)}{\beta},\sup_{\partial\Omega}u\right\}$ on $\Omega$. 
	Consider now the function $v=-u$, which clearly satisfies
	\[
	\div\left(\frac{\nabla v}{\sqrt{1+|\nabla v|^2}}\right)=b(x)v.
	\]
	with $\sup_{\partial\Omega}v=-\inf_{\partial\Omega}u$. 
	Then we also obtain that $v=-u\leq \max\left\{\frac{c_\mu(\Omega)}{\beta},-\inf_{\partial\Omega}u\right\}$ on $\Omega$, that is, $u\geq -\max\left\{\frac{c_\mu(\Omega)}{\beta},-\inf_{\partial\Omega}u\right\}=
	\min\left\{-\frac{c_\mu(\Omega)}{\beta},\inf_{\partial\Omega}u\right\}$. Summing up we have
	\[
	\sup_\Omega u \leq \max\left\{ \frac{c_\mu(\Omega)}{\beta}, \sup_{\partial\Omega}u \right\} \, , \qquad \inf_\Omega u \geq \min\left\{ -\frac{c_\mu(\Omega)}{\beta}, \inf_{\partial\Omega}u \right\}
	\]
\end{proof}
Of course, when $c_\mu(\Omega)=0$ and $\partial\Omega\neq\emptyset$ from the above theorem we deduce that the only solution of \eqref{CSE_bis} on $\Omega$ with $u|_{\partial\Omega}\equiv c=\text{constant}$ has constant sign and satisfies $|u|\leq|c|$ on $\Omega$. In particular, in that case the only solution with $u|_{\partial\Omega} \equiv 0$ is $u\equiv 0$ on $\Omega$.

We now consider the case of a graph with prescribed mean curvature $H(x)$. If the graph $\Gamma_u$ is given by $\Gamma_u(x) = (x,u(x))$ then $u$ has to satisfy
\begin{equation} \label{MCE}
\div\left(\frac{\nabla u}{\sqrt{1+|\nabla u|^2}}\right) = m H(x) \qquad \text{on } \, M \, .
\end{equation}

\begin{theorem}\label{th:pmc1}
	Let $(M,\metric)$ be an $m$-dimensional complete noncompact Riemannian manifold and let $u:M\rightarrow\R$ define an entire graph with positive mean curvature $H(x)$ satisfying, for some parameters $\mu\in[0,1]$ and $\beta>0$,
	\begin{equation} \label{Hcond}
		\left\{
		\begin{array}{l@{\quad}l@{\qquad}l}
			H(x) \geq \dfrac{\beta}{m} & \forall \, x \in M & \text{if } \, \mu = 0 \\[0.3cm]
			H(x) \geq \dfrac{\beta}{m} r(x)^{-\mu} & \text{for } \, r(x) \gg 1 & \text{if } \, 0 < \mu \leq 1 \, .
		\end{array}
		\right.
	\end{equation}
	Then 
	\[
		\beta\leq c_\mu(M)
	\]
	where $c_\mu(M)$ is defined in (\ref{cmu_M}).
\end{theorem}

The proof of the theorem is a straightforward application of Theorem \ref{A_1w_base} using the mean curvature equation (\ref{MCE}), with the choices $f(u)=1$ and $b(x)=mH(x)$. 

\begin{corollary}\label{coro_pmc1}
	Let $(M,\metric)$ be an $m$-dimensional complete noncompact Riemannian manifold. If $\beta>c_\mu(M)$ for some $\mu\in[0,1]$, then there is no global solution $u:M\rightarrow\R$ of the prescribed mean curvature equation (\ref{MCE}) for a positive function $H\in\mathcal C(M)$ satisfying \eqref{Hcond}.
\end{corollary}

\begin{remark}
	\emph{In the special case $H$ is a positive constant we can take $\mu=0$ and $\beta=mH$. Then Corollary \ref{coro_pmc1} tells us that there exists no entire graph of constant mean curvature $H$ over $M$ if
	\[
		H > \frac{1}{m} \liminf_{R\to+\infty} \frac{\log\vol(B_R)}{R} = \frac{c_0(M)}{m} \, .
	\]
	In particular there exists no such graph for any $H\neq 0$ if $c_0(M) = 0$, which is the case, for instance, if $M$ has at most polynomial volume growth in view of Remark \ref{remark_cmu}. Thus if $(M,\metric) = (\R^m,\metric_{can})$ we recover a well-known result of Chern, \cite{ch65}, and Flanders, \cite{fl66}, that is, there are no entire graphs on $\R^m$ with positive constant mean curvature $H$; while for $(M,\metric) = (\HH^m,\metric_{\HH^m})$ there are no entire graphs on $\HH^m$ with constant mean curvature $H$ satisfying
	\[
		H > \frac{m-1}{m}
	\]
	On the other hand, it is well known, see \cite{sal89} and \cite{adlr20}, that there exist entire graphs of constant positive mean curvature $H$ for any
	\[
		H \leq \frac{m-1}{m} \, .
	\]
	This shows that Corollary \ref{coro_pmc1} is sharp. We also observe that the necessary condition
	\[
		H \leq \frac{1}{m} c_0(M)
	\]
	for the existence of a graph of constant positive mean curvature $H$ is somehow related to that of Salavessa, \cite{sal89}, that is,
	\[
		H \leq \frac{1}{m} \mathfrak h(M)
	\]
	where $\mathfrak h(M)$ is the Cheeger constant of $M$.}
\end{remark}

We shall now discuss an application of Theorem \ref{A_1w_base_W}. Let $(\PP,\metric_{\PP})$ be an $m$-dimensional Riemannian manifold and let $h\in C^\infty(\PP)$, $h>0$ on $\PP$. We can then consider $\overline{M} = \RR {}_h\times\PP$, that is, the product manifold $\RR\times\PP$ with the warped product metric $\overline{\metric}$ given by
\begin{equation}
	\overline{\metric} = h(x)^2 \di t^2 + \metric_{\PP}
\end{equation}
where $t$ is the coordinate on $\R$. For instance, $\HH^{m+1}$, the hyperbolic space of constant sectional curvature $-1$ and dimension $m+1$, admits the description $\HH^{m+1} = \R {}_h\times \HH^m$ with $h(x) = \cosh(r(x))$, where $r(x) = \dist(x,o)$ in $\HH^m$ for some fixed origin $o\in\HH^m$. Thus the metric in $\HH^{m+1}$ is represented in the form
\[
\overline{\metric} = \cosh^2(r(x)) \di t^2 + \metric_{\HH^m} \, .
\]
Given a smooth function $u : \PP \to \RR$ we consider the graph $\psi$ of $u$ on $\PP$, that is, the imbedding
\[
\psi : \PP \to \overline{M}
\]
given by
\begin{equation} \label{psi_def}
\psi(x) = (u(x),x) \, , \qquad x \in \PP \, .
\end{equation}
Graphs of this kind are called equidistant graphs. For more details and results see for instance \cite{bmpr21} and the references therein.

It is well known that $\frac{\partial}{\partial t}$ is a Killing vector field on $\overline{M}$ and we set $\phi_t$ to indicate its flow. Letting $D$ denote the gradient on $(\PP,\metric_{\PP})$ the vector field
\begin{equation} \label{nu_psi}
\nu = \frac{1}{h\sqrt{1+h^2|Du|^2}} \left( \frac{\partial}{\partial t} - h^2(\phi_u)_\ast(Du) \right)
\end{equation}
is a unit normal vector field to the graph $\psi$. We set $H$ to denote the (normalized) mean curvature of $\psi$ in the direction of $\nu$. A computation shows the validity of the formula
\begin{equation}
\div_\PP\left( \frac{h Du}{\sqrt{1+h^2|Du|^2}} \right) + \left\langle \frac{h Du}{\sqrt{1+h^2|Du|^2}} \, , \frac{Dh}{h} \right\rangle = m H
\end{equation}
on $\PP$, that we can put into the form
\begin{equation} \label{divPH}
\div_{-\log h} \left( \frac{hDu}{\sqrt{1+h^2|Du|^2}} \right) = mH
\end{equation}
where the $\div_{-\log h} = (\div_\PP)_{-\log h}$; we have omitted the reference to $\PP$ for the sake of simplicity.

In case of \eqref{divPH}
\[
A(x,s,\xi) = \frac{h(x)\xi}{\sqrt{1+h^2(x)|\xi|^2}}
\]
and therefore the operator in \eqref{divPH} is weakly $1$-coercive with coercivity constant $k=1$. We recall the definition
\begin{equation}
c^h_\mu(\PP) = \begin{cases}
\displaystyle\liminf_{R\to+\infty} \frac{(1-\mu)\log\int_{B_R} h}{R^{1-\mu}} & \quad \text{if } \, 0 \leq \mu < 1 \\
\displaystyle\liminf_{R\to+\infty} \frac{\log\int_{B_R}h}{\log R} & \quad \text{if } \, \mu = 1 \, .
\end{cases}
\end{equation}
Applying Theorem \ref{A_1w_base_W} we have

\begin{theorem} \label{thm:pmc1W}
	Let $(\PP,\metric_\PP)$ be a complete manifold of dimension $m$, $h\in C^\infty(\PP)$, $h>0$ on $\PP$. Let $\psi$ as in \eqref{psi_def} be an equidistant graph in $\overline{M} = \RR {}_h \times \PP$ defined via the smooth function $u : \PP \to \RR$ with mean curvature $H$ in the direction of $\nu$ as in \eqref{nu_psi}. Suppose that $H>0$ and that \eqref{Hcond} is satisfied for some parameters $\mu\in[0,1]$ and $\beta>0$ on $\PP$. Then
	\begin{equation}
	c^h_\mu(\PP) \geq \beta \, .
	\end{equation}
\end{theorem}

In particular we have

\begin{corollary} \label{coro_pmc1W}
	In the assumptions of Theorem \ref{thm:pmc1W} suppose that $H$ is a positive constant. Then
	\[
	mH \leq c^h_0(\PP) = \liminf_{R\to+\infty} \frac{\log\int_{B_R} h}{R} \, .
	\]
\end{corollary}

It is immediate to deduce some non-existence results from Corollary \ref{coro_pmc1W}.

We now give some applications of Theorem \ref{divX_1w}. It gives an extended version of Theorem 3.1 of \cite{acdn21}. Towards this aim we let $(\overline{M},\metric)$ be an oriented Riemannian manifold of dimension $m+1$ endowed with a Killing vector field $Y$. We let $\psi : M \to \overline{M}$ be an isometric immersion of a connected, orientable, complete non-compact Riemannian manifold $M$ of dimension $m$ into $\overline{M}$. We choose a globally defined unit normal vector field $\nu$ and let $v = \langle\nu,Y\rangle$. Let $\{\theta^a\}$, $\{\theta^a_b\}$, $1\leq a,b,\dots\leq m+1$ be a local orthonormal oriented Darboux coframe along $\psi$ with respective Levi-Civita connection forms. The dual frame is therefore given by $\{e_i\}$, $i=1,\dots,m$ tangent to $M$ and $e_{m+1}=\nu$. Then
\[
v = Y^{m+1}
\]
and
\[
\di v = \di Y^{m+1} = Y^{m+1}_i \theta^i - Y^t \theta^{m+1}_t = (Y^{m+1}_i - Y^t A_{ti}) \theta^i
\]
where $A = A_{ij} \theta^i\otimes\theta^j$ is the second fundamental form in the direction of $\nu$. Note that, since $Y$ is Killing, we have
\begin{equation}
v_i = - (Y^i_{m+1} + Y^t A_{ti}) \, .
\end{equation}
Set $X = -\nabla v$. We have
\begin{equation} \label{vXsign}
\langle\nabla v,X\rangle = - |X|^2 \leq 0 \, .
\end{equation}
From equation (2.8) in Proposition 2.2 of \cite{mmr15} (see also Proposition 1 in \cite{fr04}), using the fact that $Y$ is Killing we deduce
\begin{equation} \label{jac_eq}
\Delta v = -[|A|^2 + \overline{\Ricc}(\nu,\nu)] v - m \langle Y,\nabla H\rangle
\end{equation}
where $\overline{\Ricc}$ is the Ricci tensor of $\overline{M}$; we refer the reader to Proposition 6 in \cite{adr07} for an extension of \eqref{jac_eq} to the case where $Y$ is conformal Killing. Observe that, setting $Z$ for the vector field of components
\[
Z_i = Y^t A_{ti}
\]
an immediate computation gives
\[
|X|^2 = |\nabla_\nu Y|^2 + 2 \langle Z,\nabla_\nu Y\rangle + |Z|^2
\]
so that $X$ is bounded anytime $Y$, $\nabla_\nu Y$ and $A$ are bounded.

\begin{theorem} \label{ACD_gen}
	Let $(M,\metric)$ be a complete, non-compact manifold of dimension $m$ isometrically immersed into a manifold $(\overline{M},\metric)$ of dimension $m+1$ with constant mean curvature. Let $M$ and $\overline{M}$ be orientable and let $Y$ a bounded Killing vector field on $\overline{M}$. Assume that there exists $\alpha>0$ such that
		\begin{equation} \label{alpha_b}
		\sup_M |Y| \leq \frac{1}{\alpha} \qquad \text{and} \qquad \overline{\Ricc}(\nu,\nu) > - \alpha \, .
		\end{equation} where $\nu$ is a chosen globally defined unit normal on $M$ and $\overline{\Ricc}$ is the Ricci tensor of $\overline{M}$. We let $A$ be the second fundamental form of $\psi$ in the direction of $\nu$. Let $v=\langle Y,\nu\rangle$ satisfy
	\begin{equation} \label{v_lb}
	v \geq \frac{1}{|A|^2 + \overline{\Ricc}(\nu,\nu) + \alpha} \qquad \text{on } \, M
	\end{equation}
	and assume that the vector field
	\begin{equation}
	X = \overline{\nabla}_\nu Y + A(Y^\top,\;)^\sharp = \overline{\nabla}_\nu Y + Z
	\end{equation}
	has bounded length. If $c_0(M) = 0$ then $|A|^2 = -\overline{\Ricc}(\nu,\nu)$. In particular, if $\overline{\Ricc}(\nu,\nu) \geq 0$ then $\psi(M)$ is totally geodesic.
\end{theorem}

\begin{proof}
	Since $H = \frac{1}{m} \mathrm{tr}A$ is constant, \eqref{jac_eq} yields
	\begin{equation} \label{jac_0}
	\Delta v = -[|A|^2 + \overline{\Ricc}(\nu,\nu)] v \qquad \text{on } \, M \, ,
	\end{equation}
	where $v = \langle\nu,Y\rangle\leq |Y| \leq \frac{1}{\alpha}$ since $v^2 = |Y|^2 - |Y^\top|^2 \leq |Y|^2$. Define $u=\frac{1}{\alpha}-v \geq 0$ on $M$ and observe that, because of \eqref{vXsign}
	\begin{equation}
	\langle\nabla u,X\rangle \geq 0 \qquad \text{on } \, M \, .
	\end{equation}
	From \eqref{alpha_b}, \eqref{v_lb} and \eqref{jac_0} we get
	\[
	\div X = - \Delta v = [|A|^2 + \overline{\Ricc}(\nu,\nu) + \alpha] v - \alpha v \geq 1 - \alpha v = \alpha u \, .
	\]
	Then we apply Theorem \ref{divX_1w} with the choices $\Omega = M$, $h=0$ or, equivalently, $w=1$, $\mu=0$, $\beta=\alpha>0$, $k=1$ and $f(t)=t$ to obtain that
	\[
	u \leq \frac{c_0(M)}{\alpha} \, .
	\]
	If $c_0(M)=0$ then $u\equiv0$, hence $v\equiv\frac{1}{\alpha} > 0$ and substituting into \eqref{jac_0} we get
	\[
	|A|^2 + \overline{\Ricc}(\nu,\nu) \equiv 0
	\]
	and if further $\overline{\Ricc}(\nu,\nu) \geq 0$ then $A\equiv 0$, so that $\psi$ is totally geodesic.
\end{proof}

The next result is in the same vein. Let $\psi:M\to\R^{m+1}$ be an oriented, $m$-dimensional mean curvature flow soliton with respect to the position vector field $Z$ of $\R^{m+1}$ with soliton constant $c\in\R$. This means that the isometric immersion $\psi$ satisfies the equation
\begin{equation} \label{Hsol}
m\mathbf{H} = c Z_\psi^\bot = c \psi^\bot
\end{equation}
where $\mathbf{H}$ is the mean curvature vector of the immersion and $Z_\psi^\bot = \psi^\bot$ is the normal part of the vector field $Z$ along $\psi$. We let
\begin{equation}
\nu = \nu_\psi : M \to \Sph^m \subseteq \R^{m+1}
\end{equation}
be the spherical Gauss map of the immersion and $B$ a Killing vector field on $\R^{m+1}$ and define
\[
\eta(x) = \frac{1}{2}|\psi|^2(x) \, .
\]

The following computation can be found in Proposition 1 of \cite{cmr20}.

\begin{lemma} \label{lem_B}
	In the above setting, let $v=\langle B,\nu\rangle$. Then $v$ satisfies the equation
	\begin{equation} \label{jac_B}
	\Delta_{-cZ^\top} v + |A|^2 v = - \langle[B,cZ],\nu\rangle + cv
	\end{equation}
	where $A$ is the second fundamental form in the direction of $\nu$, $\Delta_{-cZ^\top}$ is the operator $\Delta + \langle cZ^\top,\nabla\;\rangle$ and $[\;,\;]$ the Lie bracket.
\end{lemma}

\begin{theorem} \label{thm_soliton}
	Let $\psi:M\to\R^{m+1}$ be an oriented, complete, non-compact, $m$-dimensional mean curvature flow soliton with respect to the position vector field $Z$ of $\R^{m+1}$ and with soliton constant $c\in\R$. Assume that $c_0^{w}(M) = 0$, where $w=e^{c\eta}$, that $|A|$ is bounded and that the spherical Gauss map $\nu:M\to\Sph^m$ of the immersion satisfies
	\begin{equation} \label{vA_bound}
	v = \langle b,\nu\rangle \geq \frac{1}{1+|A|^2}
	\end{equation}
	for some fixed vector $b\in\Sph^m$. Denote by $B$ the parallel extension of $b$ to all of $\R^{m+1}$. Then $\psi(M)$ is a hyperplane orthogonal to $B$, and if $c\neq 0$ we have $0_{\R^{m+1}} \in \psi(M)$.
\end{theorem}

\begin{remark}
	\emph{Note that \eqref{vA_bound} implies that the image $\nu(M)$ of the spherical Gauss map of the immersion is contained in the compact subset
	\[
	\left\{y\in\Sph^m : \dist(b,y) \leq \cos^{-1}\left(\frac{1}{1+\sup_M |A|^2}\right) \right\}
	\]
	of the open hemisphere of $\Sph^m$ centered at $b$.}
\end{remark}

\begin{proof}
	First observe that we can apply Lemma \ref{lem_B} since $B$ is parallel, hence Killing, on $\R^{m+1}$. In this case we have $[B,Z] = B$, as it can be seen by fixing cartesian coordinates $(x^i)$ on $\R^{m+1}$ such that $B=\partial_{x^1}$ and computing
	\[
	[B,Z] = \left[\partial_{x^1},\sum_{i=1}^{m+1} x^i \partial_{x^i}\right] = \sum_{i=1}^{m+1} \frac{\partial x^i}{\partial x^1} \partial_{x^i} = \partial_{x^1} = B \, ,
	\]
	so that formula \eqref{jac_B} reduces to
	\begin{equation} \label{jac_1}
	\Delta_{-cZ^\top} v + |A|^2 v = 0 \, .
	\end{equation}
	Furthermore, setting $\bar\eta(x) = \frac{1}{2}|x|^2$ for $x\in\R^{m+1}$ we have $Z = \nabla\bar\eta$ in $\R^{m+1}$, and since $\eta=\bar\eta\circ\psi$ it follows that $Z^\top = \nabla\eta$. Hence $\Delta_{-cZ^\top}$ is in fact a weighted Laplacian
	\[
	\Delta_{-cZ^\top} = \Delta_{-c\eta} = \Delta + c\langle\nabla\eta,\;\rangle
	\]
	and using \eqref{jac_1} we get
	\[
	\Delta_{-c\eta} v = -|A|^2 v \, .
	\]
	As in Theorem \ref{ACD_gen}, and since $B$ is parallel, we have $\nabla v = -AB^\top$. Setting $X = AB^\top$, we can further restate the above relation as
	\begin{equation}
	\div_{-c\eta} X = |A|^2 v \, .
	\end{equation}
	Thus for $u=1-v\geq0$ (note that $v\leq|B|=1$) we have, using \eqref{vA_bound}
	\begin{equation}
	\div_{-c\eta} X \geq u \, .
	\end{equation}
	We note that $X$ is bounded since and $|B^\top|\leq|B|=1$. Also, $\langle\nabla u,X\rangle = -\langle\nabla v,X\rangle = |X|^2 \geq 0$. We can then apply Theorem \ref{divX_1w} as in the proof of Theorem \ref{ACD_gen} to deduce
	\[
	u \leq c_0^{w}(M) = 0
	\]
	and since $u\geq 0$ we conclude $u\equiv 0$, that is, $v\equiv 1$. Since $0<v=\langle B,\nu\rangle=|B^\bot|\leq|B|\equiv 1$, we must have $B^\bot=B$ along $\psi(M)$, that is, $\nu=B$. So $\psi(M)\subseteq\R^{m+1}$ is a hyperplane perpendicular to $B$. Moreover, if $c\neq 0$ then by \eqref{Hsol} it must be $Z^\bot_\psi \equiv \psi^\bot \equiv 0$, so that $\psi(M)$ is a hyperplane passing through the origin of $\R^{m+1}$.
\end{proof}

We observe that the condition $c_0^{w}(M) = 0$, with $w=e^{c\eta}$, means
\[
\liminf_{R\to+\infty} \frac{1}{R} \log\int_{B_R} e^{\frac{c}{2}|x|^2} = 0 \, .
\]
Note that when $c<0$ the above condition allows a growth of the volume of geodesic balls which is faster than polynomial. Thus the assumptions of the above theorem do not imply those of Theorem 3.1 of \cite{dxy16}. See also Theorem 2.2 of \cite{cmr20}. 

\section{Comparison for strictly monotone $1$-coercive operators}

Let $(M,\metric)$ be a complete Riemannian manifold. Let $A : \RR\times TM \to TM$ be a weakly-$1$-coercive generalized bundle map and let $L$ be the differential operator weakly defined by
$$
Lu(x) = \div A(x,u,\nabla u) \, .
$$
In this section we assume that $A$ is also \emph{strictly monotone}, that is, that
\begin{equation} \label{str_mon_1}
\langle A(x,t,\xi) - A(x,s,\eta), \xi-\eta \rangle \geq 0
\end{equation}
for every $x\in M$, $\xi,\eta\in T_x M$ and $t,s\in\RR$ with $t\geq s$, and we require strict inequality in case $t>s$ and $\xi\neq\eta$.

As an example, let
\[
	A(x,s,\xi) = \frac{\xi}{\sqrt{1+|\xi|^2}} \, .
\]
Then, by the M\={i}kljukov inequality in (see \cite{mik79}, page 165)
\begin{align*}
	\langle A(x,t,\xi) - A(x,s,\eta), \xi-\eta \rangle & = \left\langle \frac{\xi}{\sqrt{1+|\xi|^2}} - \frac{\eta}{\sqrt{1+|\eta|^2}}, \xi-\eta \right\rangle \\
	& \geq \frac{1}{2} [\sqrt{1+|\xi|^2} + \sqrt{1+|\eta|^2}] \left| \frac{\xi}{\sqrt{1+|\xi|^2}} - \frac{\eta}{\sqrt{1+|\eta|^2}} \right|^2 \, .
\end{align*}
Hence, \eqref{str_mon_1} holds, and if $\xi\neq\eta$ the inequality is clearly strict. Thus, in particular our subsequent results can be applied to the mean curvature operator.

\begin{theorem} \label{1_comp}
	Let $(M,\metric)$ be a complete manifold, $A : \RR \times TM \to TM$ a $\mathcal{C}^1$ strictly monotone weakly-$1$-coercive generalized bundle map with coercivity constant $k>0$ and let $L$ be the corresponding weakly-$1$-coercive operator.
	
	Let $b\in\mathcal C(M)$ be positive and satisfy condition \eqref{condB} for two given parameters $\mu\in[0,1]$ and $\beta>0$, and let $f\in\mathcal C(\R)$ be such that
	\begin{equation} \label{fast}
		f(s) - f(t) \geq \alpha (s-t) \qquad \forall \, s,t\in\R \, \text{ with } \, s > t
	\end{equation}
	for some constant $\alpha>0$. If $u,v\in\mathcal C^2(M)$ satisfy
	\begin{equation}
		Lu \geq b(x) f(u) \, , \qquad Lv \leq b(x) f(v) \qquad \text{in } \, M
	\end{equation}
	then
	\begin{equation} \label{cmu_comp0}
	u \leq v + 2\frac{k}{\alpha\beta} c_\mu(M) \qquad \text{in } \, M \, .
	\end{equation}
\end{theorem}

\begin{proof}
	We let $w=u-v$ and we suppose by contradiction that \eqref{cmu_comp0} is false so that, having set $c^\ast = 2\frac{k}{\alpha\beta} c_\mu(M)$ we have $\{x\in M : w(x)>c^\ast\} \neq \emptyset$. Thus choose $\eps>0$ small enough so that the set
	$$
	\Omega_\eps = \{ x \in M : w(x) > c^\ast + \eps \} \neq \emptyset \, .
	$$
	$\Omega_\eps$ is open by continuity of $w$. For short write $A_u=A(x,u,\nabla u)$, $A_v=A(x,v,\nabla v)$ and consider the vector field $X : M \to TM$ given by
	$$
	X = A_u - A_v \, .
	$$
	From monotone $1$-coercivity \eqref{str_mon_1} of $A$ we have
	\begin{equation} \label{monX}
	\langle X,\nabla w \rangle \geq 0 \qquad \text{on } \, \{w\geq0\} \, .
	\end{equation}
	The assumptions on $A$, $u$ and $v$ ensure that $X$ is of class $\mathcal C^1$ and that
	\[
	|X| = |A_u-A_v| \leq |A_u| + |A_v| \leq 2k \quad \text{ on } M \, .
	\]
	We observe that $X$ has a well-defined flow
	\begin{equation} \label{flow5}
	\Psi : [0,+\infty) \times \Omega_\eps \to \Omega_\eps.
	\end{equation}
	Indeed, since $M$ is complete and $X$ is bounded on $M$, for any $x\in\Omega_\eps$ the initial value problem
	\begin{equation} \label{CpXbis}
	\begin{cases}
	\gamma(0) = x \\
	\dot\gamma(t) = X_{\gamma(t)}
	\end{cases}
	\end{equation}
	has a (unique) solution $\gamma : [0,+\infty) \to M$. For all $t\in[0,+\infty)$ we have
	\begin{equation} \label{wX}
		\frac{\di}{\di t} w(\gamma(t)) = \langle \nabla w, \dot\gamma(t) \rangle_{\gamma(t)} = \langle \nabla w, X_{\gamma(t)} \rangle_{\gamma(t)} \, .
	\end{equation}
	We show that $\gamma(t)\in\Omega_\eps$ for all $t\in[0,+\infty)$. Since $w(\gamma(0))>0$, we also have $w(\gamma(t))>0$ for all sufficiently small $t>0$. Let
	\[
	\tilde T_x = \sup\{T\in(0,+\infty] : w(\gamma(t))>0 \text{ for all } t \in [0,T) \} \, .
	\]
	We know that $0 < \tilde T_x \leq +\infty$. Suppose by contradiction that $\tilde T_x < +\infty$. Then it must be $w(\gamma(\tilde T_x)) \leq 0$. But then $\gamma(t) \in \{ w>0 \}$ for all $t\in[0,\tilde T_x)$, so by \eqref{wX} and \eqref{monX} we have
	\[
	\frac{\di}{\di t} w(\gamma(t)) \geq 0 \qquad \forall \, t \in [0,\tilde T_x)
	\]
	so that $[0,\tilde T_x) \ni t \mapsto w(\gamma(t))$ is non-decreasing and therefore by continuity
	\[
		w(\gamma(\tilde T_x)) = \lim_{t\to \tilde T_x} w(\gamma(t)) \geq w(\gamma(0)) > 0 \, ,
	\]
	contradiction. So $\tilde T_x = +\infty$ and $\gamma([0,+\infty)) \subseteq \{w>0\}$. We just observed that $[0,+\infty) \ni t \mapsto w(\gamma(t))$ is non-decreasing, so in fact $w(\gamma(t)) \geq w(\gamma(0)) = w(x) > c^\ast + \eps$ for every $t\in[0,+\infty)$. Thus $\gamma(t)$ lies in $\Omega_\eps$ for every $t\in[0,+\infty)$.
	
	Now fix $p\in\Omega_\eps$ and $\delta>0$ such that $B_\delta(p) \subseteq \Omega_\eps$ and $\partial B_\delta(p)$ is smooth. For every $t\in[0,+\infty)$ let
	\[
		U_t = \Psi(t,B_{\delta}(p)) \qquad \text{and} \qquad g(t) = \vol(U_t) \, .
	\]
	Using that $\mathcal L_X(\di v) = \div X \, \di v$, with $\mathcal L_X$ the Lie derivative with respect to $X$, we have
	\begin{equation} \label{g'_bis}
		\begin{split}
			g'(t) = \int_{U_t} \div X \, \di v = \int_{U_t} (Lu-Lv) \, \di v & \geq \int_{U_t} b(x) [f(u)-f(v)] \, \di v \\
			& = \int_{U_t} b(x) \frac{f(u)-f(v)}{u-v} (u-v) \, \di v \\
			& > (c^\ast+\eps) \alpha \int_{U_t} b(x) \, \di v \, ,
		\end{split}
	\end{equation}
	where have used that $u-v=w>c^\ast+\eps>0$ in $\Omega_\eps \supseteq U_t$ and that
	\[
		\frac{f(u)-f(v)}{u-v} \geq \alpha
	\]
	by \eqref{fast}. Reasoning as in the proof of Theorem \ref{A_1w_base} (see \eqref{inf_b_M}), fixing $\beta^\ast\in\RR$ such that
	\[
		\beta > \beta^\ast > \frac{c^\ast}{c^\ast+\eps}\beta
	\]
	there exists $R_0>0$ such that
	\begin{equation} \label{inf_b_bis}
		\inf_{B_R} b \geq \beta^\ast R^{-\mu} \qquad \forall \, R\geq R_0 \, .
	\end{equation}
	As a consequence of $|X|\leq 2k$ and of the triangle inequality, we have
	\begin{equation} \label{Ut_in_bis}
		U_t \subseteq B_{\delta+2kt}(p) \subseteq B_{r(p)+\delta+2kt} \qquad \forall \, t \geq 0 \, .
	\end{equation}
	Setting $R_1 = \max\{R_0,r(p)+\delta\}$, from \eqref{g'_bis}, \eqref{inf_b_bis} and \eqref{Ut_in_bis} we get
	\[
		g'(t) \geq (c^\ast+\eps) \alpha\beta^\ast (R_1+2kt)^{-\mu} g(t) \qquad \forall \, t \geq 0 \, .
	\]
	Integrating over $[0,t]$
	\begin{equation} \label{g'2_bis}
		\log (g(t)) \geq \log(g(0)) + (c^\ast+\eps) \alpha\beta^\ast \int_0^t (R_1+2ks)^{-\mu} \, \di s \, .
	\end{equation}
	Note that $U_t \subseteq B_{\delta+2kt}(p) \subseteq B_{R_1+2kt}$ and also $U_t\subseteq\Omega_\eps$, we have $g(t) \leq \vol(B_{R_1+2kt})$, so with a change of integration variable $y=R_1+2ks$ in \eqref{g'2_bis} we get
	\[
		\log \vol(B_R) \geq \log(g(0)) + \frac{1}{2} (c^\ast+\eps) \alpha\beta^\ast \int_{R_1}^R y^{-\mu} \, \di y \qquad \forall \, R\geq R_1
	\]
	and therefore
	\begin{alignat*}{2}
		\log\vol B_R(o) & \geq \log g(0) + (c^\ast+\eps) \frac{\alpha\beta^\ast}{2k} \frac{R^{1-\mu}-R_1^{1-\mu}}{1-\mu} & \qquad \text{if } \, 0 \leq \mu < 1 \\
		\log\vol B_R(o) & \geq \log g(0) + (c^\ast+\eps) \frac{\alpha\beta^\ast}{2k} (\log R-\log R_1) & \qquad \text{if } \, \mu = 1 \, .
	\end{alignat*}
	This yields
	\[
		c_\mu(M) \geq (c^\ast+\eps) \frac{\alpha\beta^\ast}{2k} > c^\ast \frac{\alpha\beta}{2k} = c_\mu(M) \, ,
	\]
	contradiction.
\end{proof}

Similarly, one proves the local version by adapting the proof of Theorem \ref{1_comp} in the same way as we proved Theorem \ref{A_1w_base_bis} adapting the proof of Theorem \ref{A_1w_base}.

\begin{theorem} \label{1_comp_loc}
	Let $(M,\metric)$ be complete, $\Omega\subseteq M$ a non-empty open set, $A : \RR \times TM|_\Omega \to TM|_\Omega$ a $\mathcal{C}^1$ strictly monotone weakly-$1$-coercive generalized bundle map with coercivity constant $k>0$ and let $L$ be the corresponding operator.
	
	Let $b\in\mathcal C(M)$ be positive and satisfy \eqref{condB} for some $\mu\in[0,1]$ and $\beta>0$, and let $f\in\mathcal C(\R)$ be such that
	\[
		f(s) - f(t) \geq \alpha (s-t) \qquad \forall \, s,t\in\R \, \text{ with } \, s > t
	\]
	for some constant $\alpha>0$. If $u,v\in\mathcal C(\overline{\Omega})\cap \mathcal C^2(\Omega)$ satisfy
	\[
		\begin{cases}
			Lu \geq b(x) f(u) & \quad \text{in } \, \Omega \\
			Lv \leq b(x) f(v) & \quad \text{in } \, \Omega 
		\end{cases}
	\]
	then
	\begin{equation} \label{cmu_comp}
		u \leq v + \max\left\{\frac{2k}{\alpha\beta} c_\mu(\Omega),\sup_{\partial\Omega}(u-v)\right\} \qquad \text{in } \, \Omega.
	\end{equation}
\end{theorem}

\begin{remark}
	\emph{Theorem \ref{1_comp_loc} generalizes a comparison principle for the capillarity equation in the Euclidean space $M = \RR^m$ first proved (in case $\Omega\subseteq\RR^m$ is bounded) by Concus and Finn, \cite{cf74b}, and later extended (to possibly unbounded $\Omega$) in Finn and Hwang, \cite{fh89}. The case of the capillarity equation corresponds to $A(\xi) = (1+|\xi|^2)^{-1/2}\xi$, $f(t) = \alpha t$, $b(x)\equiv 1$. Note that in this setting one can take $\mu=0$, and since $c_0(\Omega)=0$ for any $\Omega\subseteq\RR^m$ the conclusion \eqref{cmu_comp} reduces to
	\[
		\sup_\Omega(u-v) \leq \sup_{\partial\Omega}(u-v)_+ \, .
	\]
	See also \cite{prs02} Theorems 1.3, 1.4 and 1.5 for further comparison results in the case where $M$ is a general Riemannian manifold.}
\end{remark}

\begin{remark}
	\emph{Theorems \ref{1_comp} and \ref{1_comp_loc} can be also deduced directly as applications of Theorem \ref{divX_1w} with the choices $X = A(x,u,\nabla u) - A(x,v,\nabla v)$ and $\psi = w = u-v$. Indeed, we already observed in the proof of Theorem \ref{1_comp} that $|X|\leq 2k$ and $\langle\nabla\psi,X\rangle\geq0$ due to weak coercivity and strict monotonicity of $A$, and by the assumptions on $f$ we have
	\[
	\div X = Lu - Lv \geq b(x)[f(u)-f(v)] \geq \alpha b(x) (u-v) = \alpha b(x) \psi \qquad \text{on } \, \Omega_+
	\]
	where $\Omega_+ = \{x \in \Omega : u(x) > v(x)\} \equiv \{x\in\Omega : \psi(x)>0\}$. Since $c_\mu(\Omega) \geq 0$, if $\Omega_+ = \emptyset$ then $u\leq v$ everywhere in $\Omega$ and \eqref{cmu_comp} immediately follows; otherwise, $\sup_\Omega \psi = \sup_{\Omega_+}\psi > 0$ and from Theorem \ref{divX_1w} we obtain
	\[
	\sup_\Omega \psi \leq \sup_{\Omega_+} \psi \leq \max\left\{\frac{2k}{\alpha\beta} c_\mu(\Omega_+),\sup_{\partial\Omega_+}\psi\right\} .
	\]
	Since $c_\mu(\Omega_+) \leq c_\mu(\Omega)$ and $\partial\Omega_+ \subseteq (\partial\Omega) \cup \{u=v\}$, we further obtain
	\[
	\sup_\Omega \psi \leq \max\left\{\frac{2k}{\alpha\beta} c_\mu(\Omega),\sup_{\partial\Omega}\psi\right\} ,
	\]
	that is, \eqref{cmu_comp}.}
\end{remark}

Applying Theorem \ref{1_comp_loc} to the capillarity equation
\begin{equation} \label{capillOm}
	\div\left(\frac{\nabla u}{\sqrt{1+|\nabla u|^2}}\right) = b(x) u \qquad \text{on } \, \Omega \subseteq M
\end{equation}
we have

\begin{theorem} \label{capi_thm}
	Let $(M,\metric)$ be complete and $\Omega\subsetneqq M$ a non-empty open set such that $c_\mu(\Omega)=0$ for a given $\mu\in[0,1]$. Let $0<b\in C(M)$ satisfy
	$$
		\left\{
		\begin{array}{l@{\quad}l@{\qquad}l}
			b(x) \geq \beta & \forall \, x \in M & \text{if } \, \mu = 0 \\[0.3cm]
			b(x) \geq \beta r(x)^{-\mu} & \text{for } \, r(x) \gg 1 & \text{if } \, 0 < \mu \leq 1
		\end{array}
	\right.
	$$
	for some $\beta>0$.
	If $u,v$ are solutions of equation \eqref{capillOm} in $\Omega$ such that $u = v$ on $\partial\Omega$, then $u\equiv v$ in $\Omega$.
\end{theorem}

\begin{corollary}
	Let $(M,\metric)$ be a complete manifold of dimension $m$ satisfying
	\[
	\Ricc(\nabla r,\nabla r) \geq - (m-1) \frac{B^2}{(1+r^2)^{\alpha/2}}
	\]
	for some $0 < \alpha \leq 2$ and $B \geq 0$. Let $\mu\in[0,\alpha/2)$ and let $0<b\in C(M)$ satisfy
	$$
	\left\{
		\begin{array}{l@{\quad}l@{\qquad}l}
			b(x) \geq \beta & \forall \, x \in M & \text{if } \, \mu = 0 \\[0.3cm]
			b(x) \geq \beta r(x)^{-\mu} & \text{for } \, r(x) \gg 1 & \text{if } \, 0 < \mu < \frac{\alpha}{2}
		\end{array}
	\right.
	$$
	for some $\beta>0$. Let $\Omega\subsetneqq M$ be a non-empty open set and let $u,v\in\mathcal{C}^2(\Omega)$ be solutions of \eqref{capillOm} such that $u = v$ on $\partial\Omega$. Then $u\equiv v$ in $\Omega$.
\end{corollary}

\begin{proof}
	Use Theorem \ref{capi_thm} and Remark \ref{remark_cmu} to guarantee that in the above assumptions $c_\mu(\Omega) = 0$.
\end{proof}

\section{Extension to the general setting}

The aim of this last section is to show how the results proved in Section 3 can be extended to the general setting of weakly-1-coercive operators
\begin{equation} \label{Ludef}
	L u = w^{-1} \div(w A(x,u,\nabla u))
\end{equation}
with $A$ just a Carath\'eodory-type weakly-1-coercive function, $w$ is a positive, measurable and locally bounded weight, and $u$ is just continuous and locally $W^{1,1}$ regular. To do so, we first prove the following generalization of Theorem \ref{divX_1w}.

\begin{theorem} \label{Xlessreg}
	Let $(M,\metric)$ be a complete, non-compact manifold, $\Omega\subseteq M$ a non-empty open set. Let $w \in L^\infty_\loc(\Omega)$ satisfy $w>0$ a.e.~in $\Omega$ and let $0<b\in C(M)$ satisfy condition \eqref{condB} for some $\mu\in[0,1]$ and $\beta>0$. Let $X\in L^\infty(\Omega;TM)$ and $\psi\in W^{1,1}_\loc(\Omega)\cap C(\overline{\Omega})$ be such that
	\begin{equation} \label{divXweak}
		- \int_\Omega w \, \langle X,\nabla\varphi\rangle \geq \int_\Omega w \, b \, \psi \, \varphi \qquad \forall \, \varphi\in W^{1,1}_c(\Omega) , \, \varphi \geq 0
	\end{equation}
	and
	\begin{equation}
		\langle\nabla\psi,X\rangle \geq 0 \qquad \text{a.e. in } \, \Omega \, .
	\end{equation}
	Then
	\begin{equation} \label{concX}
		\sup_\Omega \psi \leq \max\left\{ \frac{\|X\|_{L^\infty(\Omega)}}{\beta} c^w_\mu(\Omega), \sup_{\partial\Omega} \psi \right\} .
	\end{equation}
	In particular, in case $\Omega=M$ we have
	\[
	\sup_M \psi \leq \frac{\|X\|_{L^\infty(\Omega)}}{\beta} c^w_\mu(M) \, .
	\]
\end{theorem}

\begin{remark}
	\emph{In case $\|X\|_{L^\infty(\Omega)} = 0$ and $c^w_\mu(\Omega) = +\infty$, we adopt the convention $0\cdot(+\infty) = 0$ to make sense of the first term in brackets on the RHS of \eqref{concX}. Note that if $\|X\|_{L^\infty(\Omega)} = 0$ then the validity of \eqref{divXweak} forces $w b \psi\leq 0$ a.e.~in $\Omega$, and thus $\psi\leq 0$ a.e.~in $\Omega$ since $w$ and $b$ are a.e.~strictly positive. Hence, in case $\|X\|_{L^\infty(\Omega)} = 0$ we necessarily have $\sup_\Omega \psi \leq 0$ and then, even in case $c^w_\mu(\Omega) = +\infty$, inequality \eqref{concX} certainly holds because its RHS is $0$.}
\end{remark}

\begin{proof}[Proof of Theorem \ref{Xlessreg}]
	Suppose, by contradiction, that \eqref{concX} is not satisfied. Then, we can find $\gamma\in\RR$ such that
	\[
	\sup_\Omega \psi > \gamma > \max\left\{ \frac{\|X\|_{L^\infty(\Omega)}}{\beta} c^w_\mu(\Omega), \sup_{\partial\Omega} \psi \right\} .
	\]
	Since the RHS of \eqref{concX} is non-negative, necessarily $\gamma>0$. We define
	\[
		\Omega_\gamma = \{ x \in \Omega : \psi(x) > \gamma \}
	\]	
	and $G : (0,+\infty) \to [0,+\infty)$ by setting
	\[
		G(t) = \vol_{-\log w}(B_t\cap\Omega_\gamma) = \int_{B_t\cap\Omega_\gamma} w \qquad \forall \, t > 0 \, .
	\]
	The set $\Omega_\gamma$ is open, non-empty and satisfies $\overline{\Omega_\gamma} \subseteq \Omega$ because $\gamma > \sup_{\partial\Omega}\psi$. The function $G$ is well defined, non-decreasing and absolutely continuous on any compact interval contained in $(0,+\infty)$. In particular, it is differentiable a.e.~in $(0,+\infty)$. Moreover, by monotone convergence together with the fact that $w>0$ a.e.~and $\vol(\Omega_\gamma)>0$, we have
	\begin{equation} \label{Glim}
		\lim_{R\to+\infty} G(R) = \int_{\Omega_\gamma} w > 0 \, .
	\end{equation}
	As in the proof of Theorem \ref{A_1w_base}, fix $\beta^\ast\in\RR$ such that
	\[
		\beta > \beta^\ast > \frac{\|X\|_{L^\infty(\Omega)}}{\gamma} c^w_\mu(\Omega)
	\]
	and let $R_0>0$ be large enough so that
	\[
	\inf_{B_R} b(x) \geq \beta^\ast R^{-\mu} \qquad \forall \, R \geq R_0 \, .
	\]
	Similarly to what we did in the proof of Theorem \ref{A_1w_base}, we aim at showing that
	\begin{equation} \label{G'G}
		\frac{\|X\|_{L^\infty(\Omega)}}{\beta^\ast} G'(t) \geq \frac{\gamma}{t^\mu} G(t) \qquad \text{for a.e.~} \, t\in(R_0,+\infty) \, .
	\end{equation}
	Provided \eqref{G'G} holds, if we choose $R_1 = R_1(\gamma) >R_0$ such that $G(R_1)>0$ (the existence of such $R_1$ is guaranteed by \eqref{Glim}) then integrating the differential inequality \eqref{G'G} on $[R_1,R]$ for any $R>R_1$ (and using the fact that $\log G$ is absolutely continuous on $[R_1,R]$ because $G$ is absolutely continuous, positive and bounded away from $0$ on that interval) we get
	\[
	\frac{\|X\|_{L^\infty(\Omega)}}{\beta^\ast} [\log G(R) - \log G(R_1)] \geq \gamma\begin{cases}
		\left(\dfrac{R^{1-\mu}}{1-\mu} - \dfrac{R_1^{1-\mu}}{1-\mu}\right) & \text{if } \, 0 \leq \mu < 1 \\[0.2cm]
		(\log R - \log R_1) & \text{if } \, \mu = 1
	\end{cases}
	\]
	for all $R>R_1$, from which we obtain, letting $R\to+\infty$,
	\[
	\frac{\|X\|_{L^\infty(\Omega)}}{\beta^\ast} c^w_\mu(\Omega) \geq \frac{\|X\|_{L^\infty(\Omega)}}{\beta^\ast} c^w_\mu(\Omega_\gamma) \geq \gamma
	\]
	thus reaching the desired contradiction. Hence, we are only left to prove validity of \eqref{G'G} under the aformentioned assumptions on $\gamma$ and $R_0$. Let $t>R_0$ be a value for which $G'(t)$ exists. For any $0<\delta<t$ choose $\eta_\delta\in C^\infty(M)$ satisfying
	\[
	\begin{cases}
		\eta_\delta \equiv 1 & \quad \text{on } \, B_{t-\delta} \\
		\eta_\delta \equiv 0 & \quad \text{on } \, M\setminus B_t \\
		0 \leq \eta_\delta \leq 1 & \quad \text{on } \, B_t \setminus B_{t-\delta} \\
		|\nabla\eta_\delta| \leq \dfrac{1}{\delta} + 1 & \quad \text{on } \, M \, .
	\end{cases}
	\]
	Let $\lambda\in C^\infty(\RR)$ be such that
	\[
	\lambda(s) = 0 \quad \text{if } \, s \leq 1 \, , \qquad \lambda(s) = 1 \quad \text{if } \, s \geq 2 \, , \qquad \lambda' \geq 0 \quad \text{on } \, \RR
	\]
	and for any $\eps>0$ define $\lambda_\eps\in C^\infty(\RR)$ by
	\[
	\lambda_\eps(s) = \lambda(s/\eps) \, .
	\]
	We have
	\[
	0 \leq \lambda_\eps \leq \mathbf{1}_{(0,+\infty)} \quad \forall \, \eps > 0 \qquad \text{and} \qquad \lambda_\eps \nearrow \mathbf{1}_{(0,+\infty)} \quad \text{as } \, \eps \to 0^+ \, ,
	\]
	where $\mathbf{1}$ stands for the indicator function and $\nearrow$ denotes monotone convergence from below. Now, for any $0<\delta<t$ and $\eps>0$ we consider the test function $\varphi_{\delta,\eps} \in W^{1,1}_c(\Omega)$ defined by
	\[
	\varphi_{\delta,\eps} = \eta_\delta \cdot \lambda_\eps(\psi-\gamma) \, .
	\]
	By construction, $\varphi_{\delta,\eps}$ is non-negative and compactly supported inside $\Omega_\gamma \subseteq \Omega$, so it is an admissible test function for the differential inequality \eqref{divXweak}, yielding
	\[
	- \int_{\Omega_\gamma} w \langle X,\nabla\varphi_{\delta,\eps}\rangle \geq \int_{\Omega_\gamma} w \, b \, \psi \, \varphi_{\delta,\eps} \, .
	\]
	We have
	\begin{align*}
		- w \langle X,\nabla\varphi_{\delta,\eps}\rangle & = - w \lambda_\eps(\psi-\gamma) \langle X,\nabla\eta_\delta\rangle - w \eta_\delta \lambda_\eps'(\psi-\gamma) \langle X,\nabla\psi\rangle \\
		& \leq - w \lambda_\eps(\psi-\gamma) \langle X,\nabla\eta_\delta\rangle
	\end{align*}
	where the inequality holds since $w \, \eta_\delta\geq0$ and $\lambda_\eps'\geq0$ by construction, and $\langle X,\nabla\psi\rangle \geq 0$ by assumption. Thus, we have
	\[
	- \int_{\Omega_\gamma} w \, \lambda_\eps(\psi-\gamma) \langle X,\nabla\eta_\delta\rangle \geq \int_{\Omega_\gamma} w \, b \, \psi \, \lambda_\eps(\psi-\gamma) \, \eta_\delta \, .
	\]
	Letting $\eps\to0^+$ and applying the dominated convergence theorem and the monotone convergence theorem to the left and right sides of this inequality, respectively, we obtain
	\[
	- \int_{\Omega_\gamma} w \langle X,\nabla\eta_\delta\rangle \geq \int_{\Omega_\gamma} w \, b \, \psi \, \eta_\delta = \int_{\Omega_\gamma \cap B_t} w \, b \, \psi \, \eta_\delta \geq \frac{\beta^\ast\gamma}{t^\mu} \int_{\Omega_\gamma} w \eta_\delta
	\]
	where in the middle equality we used the fact that $\eta_\delta = 0$ outside $B_t$, and in the last inequality we used that $\psi \geq \gamma$ on $\Omega_\gamma$ and that $b \geq \beta^\ast t^{-\mu}$ on $B_t$. By Cauchy-Schwarz inequality, we further estimate (using that $\nabla\eta_\delta$ is supported in $B_t\setminus B_{t-\delta}$ and that $\eta_\delta = 1$ on $B_{t-\delta}$)
	\[
	\|X\|_{L^\infty(\Omega)} \|\nabla\eta_\delta\|_{L^\infty(M)} \int_{\Omega\gamma \cap B_t \setminus B_{t-\delta}} w \geq \frac{\beta^\ast\gamma}{t^\mu} \int_{\Omega\gamma \cap B_{t-\delta}} w \, ,
	\]
	that is, using the upper bound $|\nabla\eta_\delta| \leq 1+1/\delta$ and the definition of $G$,
	\[
	\frac{\|X\|_{L^\infty(\Omega)}}{\beta^\ast} (1+\delta) \frac{G(t)-G(t-\delta)}{\delta} \geq \frac{\gamma}{t^\mu} G(t-\delta) \, .
	\]
	Letting $\delta\to0^+$ and using that $t$ is a differentiability point for $G$ we obtain
	\[
	\frac{\|X\|_{L^\infty(\Omega)}}{\beta^\ast} G'(t) \geq \frac{\gamma}{t^\mu} G(t) \, .
	\]
	So, we proved that this inequality is satisfied for every $t\in(R_0,+\infty)$ such that $G'(t)$ exists, so in particular \eqref{G'G} is proved and this concludes the argument.
\end{proof}

From Theorem \ref{Xlessreg} we deduce the following extensions of Theorem \ref{A_1w_base_bisW} to the general aforementioned setting.

\begin{theorem} \label{A_1w_lessreg}
	Let $(M,\metric)$ be a complete Riemannian manifold, $\Omega\subseteq M$ an open set. Let $L$ be a weakly-1-coercive differential operator weakly defined as in \eqref{Ludef} for a weakly-1-coercive Carath\'eodory-type function $A : \RR\times TM|_\Omega \to TM|_\Omega$ with coercivity constant $k>0$ and a weight $w\in L^\infty_\loc(\Omega)$ satisfying $w>0$ a.e.~in $\Omega$.
	
	Let $0<b\in\mathcal C(M)$ satisfy condition \eqref{condB} for some $\mu\in[0,1]$ and $\beta>0$ and let $f\in\mathcal C^1(\RR)$ be non-decreasing. If $u\in W^{1,1}_\loc(\Omega) \cap \mathcal C(\overline{\Omega})$ satisfies
	\[
		L u \geq b(x) f(u) \quad \text{in } \, \Omega
	\]
	then
	\[
		\sup_\Omega f(u) \leq \max\left\{\frac{k}{\beta}c^w_\mu(\Omega), \sup_{\partial\Omega}f(u) \right\} .
	\]
	In particular, in case $\Omega=M$ we have $f(u) \leq k c^w_\mu(M)/\beta$ on $M$.
\end{theorem}

\begin{proof}
	The thesis follows by direct application of Theorem \ref{Xlessreg} with the choice $X = A(x,u,\nabla u)$ and $\psi = f(u)$.
\end{proof}

\begin{remark}
	\emph{The regularity condition $f\in\mathcal{C}^1(\RR)$ ensures that $f(u) \in W^{1,1}_\loc(\Omega)$ as long as $u\in W^{1,1}_\loc(M)\cap\mathcal{C}(\overline{\Omega})$, so that we can directly apply Theorem \ref{Xlessreg} with the choice $\psi = f(u)$. In fact, an analogue of the conclusion of Theorem \ref{A_1w_lessreg} still holds also in case $f$ is only assumed to be measurable and non-decreasing. Namely, in this case one can show that
	\begin{equation} \label{fu_comp}
		\lim_{s\nearrow u^\ast} f(s) \leq \max\left\{\frac{k}{\beta}c^w_\mu(\Omega), \sup_{\partial\Omega}f(u) \right\}
	\end{equation}
	where $u^\ast = \sup_\Omega u$ and by $\lim_{s\nearrow u^\ast} f(s)$ we intend the limit of $f(s)$ as $s$ approaches $u^\ast$ from the left. This limit exists by monotonicity of $f$, and it is always less or equal than $\sup_\Omega f(u)$. The two values concide if either $u < u^\ast$ everywhere in $\Omega$, or if $u^\ast$ is actually attained at some point $x_0\in\Omega$ and $f$ is continuous at $u^\ast$. Inequality \eqref{fu_comp} can be proved by following the same reasoning used in the proof of Theorem \ref{Xlessreg}, with $\lambda_\eps(\psi-\gamma)$ replaced by $\lambda_\eps(u-u_0)$ for a suitable choice of $u_0\in\R$ (namely, if by contradiction there exists $\gamma\in\R$ such that
	\[
		\lim_{s\nearrow u^\ast} f(s) > \gamma > \max\left\{\frac{k}{\beta}c^w_\mu(\Omega), \sup_{\partial\Omega}f(u) \right\}
	\]
	then one can choose any $u_0 < u^\ast$ such that $f(u_0) > \gamma$, so that on the open, non-empty set $\Omega_0 = \{x \in \Omega : u(x) > u_0\}$ it holds $f(u) \geq f(u_0) > \gamma$; by continuity of $u$, the inequality $u\geq u_0$ and therefore $f(u) > \gamma$ holds on $\overline{\Omega_0}$ as well, implying that $\overline{\Omega_0}\subseteq\Omega$.)}
\end{remark}

Similarly, one can prove the following extension of Theorem \ref{1_comp_loc}.

\begin{theorem}
	Let $(M,\metric)$ be a complete Riemannian manifold, $\Omega\subseteq M$ an open set. Let $L$ be the differential operator defined as in \eqref{Ludef} for a strictly monotone weakly-1-coercive Carath\'eodory-type function $A : \RR\times TM|_\Omega \to TM|_\Omega$ with coercivity constant $k>0$ and a weight $w\in L^\infty_\loc(\Omega)$ satisfying $w>0$ a.e.~in $\Omega$.
	
	Let $b\in\mathcal C(M)$ be positive and satisfy \eqref{condB} for some $\mu\in[0,1]$ and $\beta>0$, and let $f\in\mathcal{C}^1(\R)$ be such that
	\[
		f(s) - f(t) \geq \alpha (s-t) \qquad \forall \, s,t\in\R \, \text{ with } \, s > t
	\]
	for some constant $\alpha>0$. If $u,v\in W^{1,1}_\loc(\Omega) \cap \mathcal C(\overline{\Omega})$ satisfy
	\[
		\begin{cases}
			Lu \geq b(x) f(u) & \quad \text{in } \, \Omega \\
			Lv \leq b(x) f(v) & \quad \text{in } \, \Omega
		\end{cases}
	\]
	then
	\begin{equation}
		u \leq v + \max\left\{\frac{2k}{\alpha\beta} c_\mu(\Omega),\sup_{\partial\Omega}(u-v)\right\} \qquad \text{in } \, \Omega.
	\end{equation}
\end{theorem}

\bigskip

\noindent \textbf{Acknowledgements} \; The authors would like to thank the anonymous referee for reading the manuscript in great detail and for giving several valuable suggestions and useful corrections to improve the paper.

\bigskip

\noindent \textbf{Competing interests} \; Luis J. Al\'{\i}as and Marco Rigoli are partially supported by the grant PID2021-124157NB-I00 funded by MCIN/AEI/10.13039/501100011033/ ‘ERDF A way of making Europe’, Spain, and they are also supported by Comunidad Aut\'{o}noma de la Regi\'{o}n de Murcia, Spain, within the framework of the Regional Programme in Promotion of the Scientific and Technical Research (Action Plan 2022), by Fundaci\'{o}n S\'{e}neca, Regional Agency of Science and Technology, REF, 21899/PI/22. Marco Rigoli is also partially supported by the MUR grant for the PRIN project 20225J97H5.

\end{document}